\documentclass[11pt, A4, leqno]{amsart}

\usepackage{enumerate}
\usepackage{graphicx}
\usepackage{physics}
\usepackage{hyperref}

\makeatletter
\def\@defaultbiblabelstyle#1{[#1]}
\makeatother

\makeatletter
\@namedef{subjclassname@2020}{\textup{2020} Mathematics Subject Classification}
\makeatother

\numberwithin{equation}{section}

\theoremstyle{plain}
  \newtheorem{theorem}{Theorem}[section]
  \newtheorem{proposition}[theorem]{Proposition}
  \newtheorem{fact}[theorem]{Fact}
  \newtheorem{lemma}[theorem]{Lemma}
  \newtheorem{corollary}[theorem]{Corollary}
\theoremstyle{definition}
  \newtheorem{definition}[theorem]{Definition}
\theoremstyle{remark}
  \newtheorem{remark}[theorem]{Remark}
  \newtheorem*{acknowledgements}{Acknowledgements}
  \newtheorem{example}[theorem]{Example}
\numberwithin{equation}{section}

\renewcommand{\theenumi}{{\rm(\arabic{enumi})}}
\renewcommand{\labelenumi}{\theenumi}

\def\labelenumi{(\theenumi)}
\def\theenumi{\arabic{enumi}}

\DeclareMathOperator\ord{ord}

\renewcommand{\phi}{\varphi}
\renewcommand{\epsilon}{\varepsilon}
\newcommand{\setm}{\,;\,}
\newcommand{\col}{\colon}
\newcommand{\R}{\mathbb{R}}
\newcommand{\I}{\mathbb{I}}
\newcommand{\E}{\mathbb{E}}
\renewcommand{\L}{\mathbb{L}}
\newcommand{\C}{\mathbb{C}}

\newcommand{\D}{\mathbb{D}}
\newcommand{\Z}{\mathbb{Z}}
\newcommand{\inner}[2]{\left\langle{#1},{#2}\right\rangle}

\allowdisplaybreaks

\title[Zero mean curvature surfaces in the isotropic 3-space]{A new framework of zero mean curvature surfaces in the isotropic 3-space}
\author[]{Riku Kishida}
\address[Riku Kishida]{Department of Mathematical and Computing Sciences, Institute of Science Tokyo, Tokyo 152-8552, Japan}
\email{kishida.r.1632@m.isct.ac.jp}

\date{April 24, 2026}

\begin{document}

\begin{abstract}
  We introduce a class of zero mean curvature surfaces with singularities in the isotropic 3-space, called \textit{ZMC-faces}.
  As a main result, we establish three Osserman-type inequalities for a ZMC-face under certain assumptions on both completeness and finiteness of the total curvature.
  The equality conditions of these inequalities are related to the asymptotic behaviors of the ends.
  Moreover, we present several examples of ZMC-faces attaining equalities in these inequalities.
\end{abstract}

\maketitle

\tableofcontents

\section*{Introduction}

The \textit{isotropic $3$-space} $\I^3$ is the 3-dimensional vector space $\R^3$ with the degenerate inner product $\inner{\cdot}{\cdot}$ defined by
\begin{align}\label{eq_degenerate_metric}
  \inner{(t_1,x_1,y_1)}{(t_2,x_2,y_2)}:=x_1x_2+y_1y_2\qquad\bigl((t_1,x_1,y_1),(t_2,x_2,y_2)\in\R^3\bigr).
\end{align}
The isotropic 3-space $\I^3$ has been regarded as an important example of Cayley-Klein geometries, and the surface theory in $\I^3$ was developed by Strubecker~\cite{Strubecker1942}.
A comprehensive account of the geometry of $\I^3$ can be found in Sachs~\cite{Sachs1990}.

Let $\Sigma^2$ be a Riemann surface, and let $f\col\Sigma^2\to\I^3$ be a space-like conformal immersion.
Then, the \textit{mean curvature function} $H$ of $f$ can be defined in a natural way (see Section~\ref{sec_1} in details).
We call $f$ the \textit{zero mean curvature surface} in $\I^3$ if $H$ is identically zero.
Let $\phi$ be a smooth real-valued function defined on the entire plane $\R^2$, and let us consider the map \[f\col\C\ni x+iy\mapsto(\phi(x,y),x,y)\in\I^3\qquad(x,y\in\R)\] as the graph of $\phi$, where $i$ is the imaginary unit.
Then, the mean curvature function of $f$ is written as \[H=\frac{1}{2}(\phi_{xx}+\phi_{yy}).\]
This implies that any zero mean curvature surface in $\I^3$ can be locally viewed as the graph of a harmonic function $\phi$.
Therefore, there exist infinitely many zero mean curvature surfaces in $\I^3$ which can be represented as entire graphs.

The Weierstrass-type representation formula for zero mean curvature surfaces in $\I^3$ is known.
In fact, given a zero mean curvature surface $f\col\Sigma^2\to\I^3$, there exist a holomorphic function $g$ and a holomorphic 1-form $\omega$ on $\Sigma^2$, such that the image of $f$ is congruent to that of
\begin{align}\label{eq_intro_Weierstrass_representation_formula}
  \Re\int(g,1,-i)\omega.
\end{align}
The pair $(g,\omega)$ is called the \textit{Weierstrass data} of $f$.
In particular, $g$ is called the \textit{Gauss map} of $f$.
As a typical example, a \textit{catenoid} in $\I^3$ can be obtained from the Weierstrass data $(g,\omega):=(1/z,dz)$ (see Figure~\ref{fig_catenoid_sing1} (a)).
The expression \eqref{eq_intro_Weierstrass_representation_formula} can be regarded as an intermediate formula between minimal surfaces in the Euclidean 3-space $\E^3$ and maximal surfaces in the Lorentz-Minkowski 3-space $\L^3$.
In light of this observation, there has been much research on extending various properties of minimal surfaces in $\E^3$ or maximal surfaces in $\L^3$ to zero mean curvature surfaces in $\I^3$ (for example, \cite{AF2022b, AF2022a, ALYpreprint, CH2025, CLLY2024, daSilva2021, Kato, Sato2021, SY2021}).

These papers indicate that zero mean curvature surfaces in $\I^3$ may have singularities.
In particular, Sato~\cite{Sato2021} and Kato~\cite{Kato} discuss the Weierstrass-type representation formula for zero mean curvature surfaces with singularities and provide some examples.
For example, we set $(g,\omega):=(1/z,zdz)$ as in \cite[Example~1]{Sato2021}, then the corresponding zero mean curvature surface
\begin{align}\label{eq_sing1}
  f(u+iv)=\left(u,\frac{1}{2}\left(u^2-v^2\right),uv\right)\qquad(u,v\in\R)
\end{align}
has a cross cap (or a Whitney's umbrella) at $(u,v)=(0,0)$ (see Figure~\ref{fig_catenoid_sing1} (b)).
Therefore, even when each component is a polynomial of degree at most two, there exist nontrivial zero mean curvature surfaces with singularities.

\begin{figure}[htbp]
  \centering
  \begin{tabular}{c@{\hspace{2cm}}c}
    \includegraphics[width=5cm]{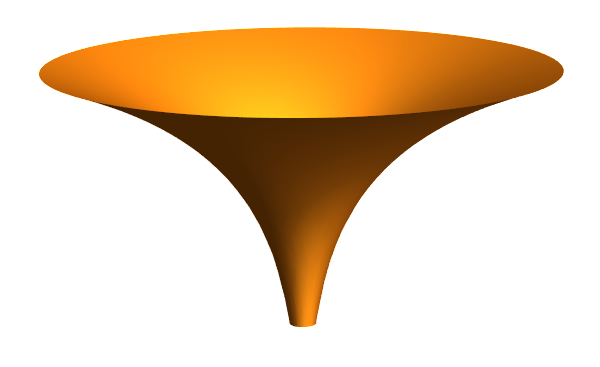} &
    \includegraphics[width=2.5cm]{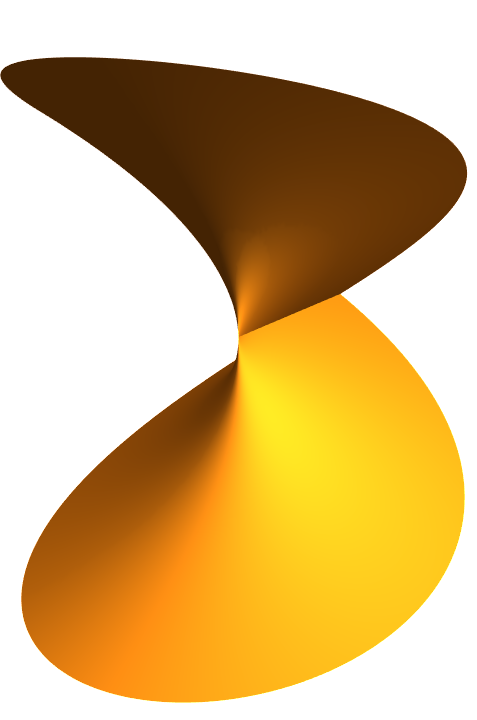} \vspace{0.3cm}\\
    (a) & (b)
  \end{tabular}
  \caption{
    (a) Left: The catenoid in $\I^3$ (the upper side is an expanding end, whereas the lower side is a shrinking end).
    (b) Right: The zero mean curvature surface given by \eqref{eq_sing1}.
    }
  \label{fig_catenoid_sing1}
\end{figure}

Based on these references, we introduce \textit{ZMC-faces} as a class of zero mean curvature surfaces in $\I^3$ allowing some kind singularities (cf. Definition~\ref{def_ZMC_face}).
This is an analogy of maxfaces in Umehara-Yamada~\cite{UY2006}, CMC-1 faces in Fujimori~\cite{Fujimori2006} and minfaces in Takahashi~\cite{Takahashi2012} and Akamine~\cite{Akamine2019}.
We show that ZMC-faces also admit a Weierstrass-type representation formula in Proposition~\ref{prop_representation_formula_ZMC_face}.
Moreover, based on the fact that there exists a certain duality for space-like surfaces in $\I^3$ (cf. Strubecker~\cite{Strubecker1978} and Pottmann-Liu~\cite{PL2007}), we discuss the ``dual surface'' obtained from a ZMC-face.
From the viewpoint of this duality, we introduce the \textit{dual Gauss map} for a ZMC-face in Definition~\ref{def_dual_surface}.

The main aim of this paper is to establish Osserman-type inequalities for ZMC-faces with or without singularities.
For this purpose, using a method similar to that of \cite{UY2006}, we shall define \textit{weak completeness} and \textit{finite-type} for ZMC-faces (cf. Definition~\ref{def_weakly_complete_and_finite_type}).
A catenoid and \eqref{eq_sing1} are examples of weakly complete and finite-type ZMC-faces.
In particular, for zero mean curvature surfaces represented as entire graphs, the following proposition holds.

\begin{proposition}\label{prop_finite_type_entire_ZMC}
  Let $\phi(x,y)$ be a harmonic function defined on the entire plane $\R^2$, and let $f\col\C\to\I^3$ be a zero mean curvature surface represented as the graph of $\phi$.
  Then, $f$ is finite-type if and only if $\phi$ is a harmonic polynomial.
\end{proposition}

The equality conditions of the already known Osserman-type inequalities in various classes of surfaces (for example, \cite{Fujimori2006, KUY2004, UY2006}) are deeply related to the behavior of the ends, especially their embeddedness.
From this viewpoint, in Theorem~\ref{thm_first} we establish an Osserman-type inequality focusing on whether the ends of a ZMC-face are embedded.
In this paper, we refer to this inequality as the \textit{first Osserman-type inequality}.
However, by observing examples of ends, one sees that ZMC-faces exhibit a distinctive phenomenon which does not appear in other classes of surfaces.
Indeed, Proposition~\ref{prop_finite_type_entire_ZMC} implies that the graph of any harmonic polynomial is an example of weakly complete and finite-type ZMC-faces, and such examples attain equality in the first Osserman-type inequality since their ends are embedded.
This indicates that the first Osserman-type inequality does not fully capture the geometric features of ZMC-faces.

Accordingly, we restrict out attention to embedded ends with simple asymptotic behavior, such as \textit{planar} or \textit{catenoidal} ends.
As can be seen from a catenoid in $\I^3$, the upper and lower sides in Figure~\ref{fig_catenoid_sing1} (a) exhibit completely different behaviors.
For this reason, in this paper we distinguish between \textit{expanding} ends as on the upper side and \textit{shrinking} ends as on the lower side (cf. Definition~\ref{def_expanding_shrinking}).
Note that Kato~\cite{Kato} refers to them as large ends and small ends, respectively.
As a result, we obtain in Theorem~\ref{thm_second} an Osserman-type inequality whose equality condition is that each end is either planar, expanding catenoidal, or shrinking catenoidal.
We refer to this inequality as the \textit{second Osserman-type inequality}.
This inequality has an expression similar to the Osserman-type inequality for flat fronts in the hyperbolic 3-space, which is established in Kokubu-Umehara-Yamada~\cite{KUY2004}.
More precisely, the second Osserman-type inequality is concerned with the sum of the degrees of the Gauss map $g$ and the dual Gauss map $g_*$.

For shrinking ends, one observes an unusual phenomenon related to embeddedness.
In fact, there exists a example whose expression is asymptotic to a shrinking end of a catenoid, but which is not necessarily embedded.
Here, we say that such ends are \textit{layered shrinking catenoidal} (cf. Definition~\ref{def_layered_shrinking_catenoidal}).
This phenomenon is also specific to ZMC-faces in $\I^3$.
Consequently, by considering on layered shrinking catenoidal ends, we obtain the \textit{third Osserman-type inequality} as follows:

\begin{theorem}[Third Osserman-type inequality]\label{thm_third}
  Let $\bar\Sigma^2$ be a compact Riemann surface and $f\col\bar\Sigma^2\setminus\{p_1,\dots,p_n\}\to\I^3$ a weakly complete and finite-type ZMC-face.
  Then, the Gauss map $g\col\bar\Sigma^2\to\C\cup\{\infty\}$ of $f$ satisfies
  \begin{align}\label{eq_third_Osserman_type_inequality_intro}
    \deg(g)\ge n+k-\chi(\bar\Sigma^2),
  \end{align}
  where $\deg(g)$ is the degree of $g$ and $\chi(\bar\Sigma^2)$ is the Euler characteristic of $\bar\Sigma^2$.
  The equality of \eqref{eq_third_Osserman_type_inequality_intro} holds if and only if each end $p_j$ is either planar, expanding catenoidal, or layered shrinking catenoidal.
  In particular, if $f$ is represented as an entire graph, then the equality condition of \eqref{eq_third_Osserman_type_inequality_intro} is equivalent that the image of $f$ is a plane in $\I^3$.
\end{theorem}

Since, by definition, a layered shrinking catenoidal end includes an embedded one, the equality condition of the third Osserman-type inequality is weaker than that of the second Osserman-type inequality.

As an application of the third Osserman-type inequality, it can be used to derive an estimate for the maximal number of omitted values of the Gauss maps for weakly complete and finite-type ZMC-faces (see Section~\ref{sec_7} for details).

Note that space-like stationary surfaces in the Lorentz-Minkowski 4-space $\L^4$ can be regarded as a wider class than zero mean curvature surfaces in $\I^3$ (cf. Al\'ias-Palmer~\cite{AP1998} and Ma-Wang-Wang~\cite{MWW2013}).
In particular, Ma-Wang-Wang~\cite{MWW2013} establishes an Osserman-type inequality for algebraic stationary surfaces.
The discussion, however, is carried out under the assumption that the surface is not a zero mean curvature surfaces in $\I^3$.
Therefore, this paper fills the blank not addressed in \cite{MWW2013}.

\begin{figure}[htbp]
  \centering
  \includegraphics[width=4.5cm]{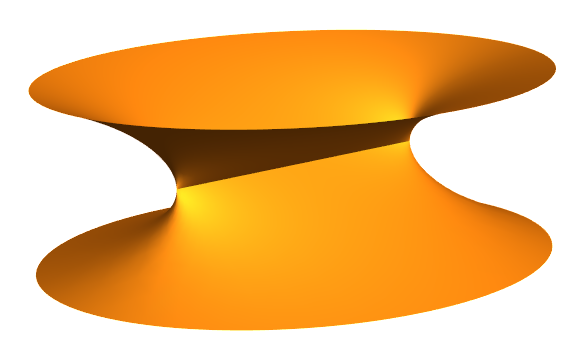}\vspace{0.3cm}\\
  (a)\vspace{0.5cm}\\
  \begin{tabular}{c@{\hspace{1.5cm}}c}
    \includegraphics[width=5cm]{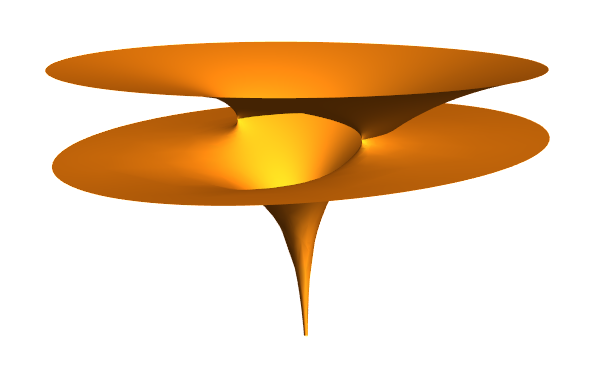} &
    \includegraphics[width=5cm]{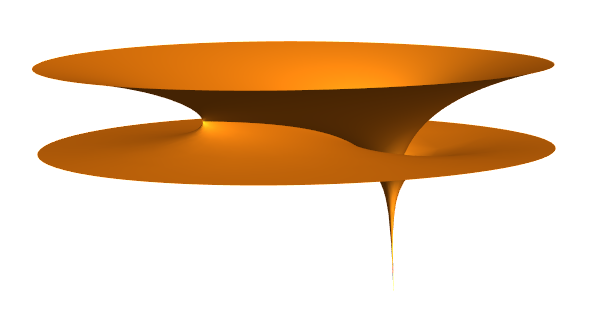}\vspace{0.3cm}\\
    (b) & (c)\vspace{0.5cm}\\
    \includegraphics[width=4cm]{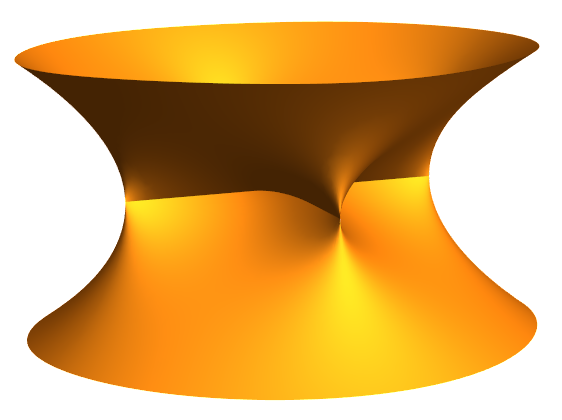} &
    \includegraphics[width=5cm]{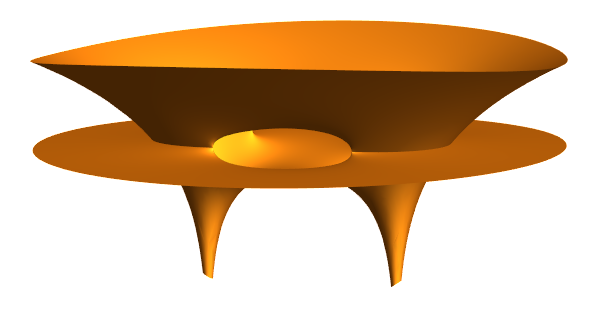} \vspace{0.3cm}\\
    (d) & (e)
  \end{tabular}
  \caption{
    Examples of weakly complete and finite-type ZMC-faces attaining equality in the Osserman-type inequalities. 
    (a) Upper: Example~\ref{exam_two_catenoidal_ends}.
    (b) Middle-left: Example~\ref{exam_same_ZMC_1}.
    (c) Middle-right: Example~\ref{exam_same_ZMC_2}.
    (d) Lower-left: Example~\ref{exam_genus_1} with $c=0$.
    (e) Lower-right: Example~\ref{exam_genus_1} with $c=\wp(1/4)$.
    Note that, the figures (d) and (e) are plotted with the $t$-axis expanded by a factor of $4$ to improve visibility.
  }
  \label{fig_intro}
\end{figure}

Moreover, in this paper we give several examples of ZMC-faces attaining equalities in the Osserman-type inequalities.
For example, the ZMC-face of Figure~\ref{fig_intro} (a) has two expanding catenoidal ends, and attains equality in all Osserman-type inequalities.
The ZMC-face of Figure~\ref{fig_intro} (b) also attains equality in all Osserman-type inequalities.
On the other hand, the ZMC-face of Figure~\ref{fig_intro} (c) has a layered shrinking catenoidal end which is not embedded.
Therefore, it satisfies the equality condition only for the third Osserman-type inequality.
These examples are of genus 0, but examples of genus 1 can also be constructed.
Indeed, the ZMC-faces of Figure~\ref{fig_intro} (d) and (e) are such examples, and attain equality in all three Osserman-type inequalities.
Furthermore, the four examples in Figure~\ref{fig_intro} (a), (b), (d) and (e) are properly embedded outside a compact subset.

The paper is organized as follows:
In Section~\ref{sec_1}, we review basic facts on zero mean curvature surfaces in the isotropic 3-space $\I^3$.
In Section~\ref{sec_2}, we define ZMC-faces and show the Weierstrass-type representation formula.
In Section~\ref{sec_3}, we introduce the dual Gauss map of a ZMC-face.
In Section~\ref{sec_4}, we discuss the weak completeness and the finiteness of the total curvature for a ZMC-face.
In Section~\ref{sec_5}, we analyze the asymptotic behaviors of ends.
In Section~\ref{sec_6}, we show the second and third Osserman-type inequalities.
In Section~\ref{sec_7}, we estimate the maximal number of omitted values of the Gauss maps.
In Section~\ref{sec_8}, we present examples which attain equalities in the Osserman-type inequalities.
In Appendix~\ref{app_A}, we give a criterion for cross caps of ZMC-faces.

\begin{acknowledgements}
  The author would like to express my gratitude to Yu Kawakami, Shin Kato, Seong-Deog Yang, Kentaro Saji, Kotaro Yamada, Atsufumi Honda and Jun Matsumoto for valuable comments and suggestions.
  The author is also grateful to Shunsuke Ichiki and Masaaki Umehara for their careful reading and helpful comments.
  This work was supported by JSPS KAKENHI Grant number JP25KJ1223.
\end{acknowledgements}

\section{Preliminaries}\label{sec_1}

In this section, we introduce some basic notations of zero mean curvature surfaces in the isotropic 3-space $\I^3$.
Here, we mainly follow Sato~\cite{Sato2021}, Seo-Yang~\cite{SY2021} and Cho-Lee-Lee-Yang~\cite{CLLY2024}.

The \textit{isotropic $3$-space} $\I^3$ is the 3-dimensional vector space $\R^3$ with the degenerate inner product $\inner{\cdot}{\cdot}$ defined by \eqref{eq_degenerate_metric}, whose signature is $(0++)$.
The isotropic 3-space $\I^3$ can be identified with the light-like hyperplane in the Lorentz-Minkowski 4-space $\L^4$ whose signature is $(-+++)$ via
\[\I^3\ni(t,x,y)\quad\longleftrightarrow\quad(t,x,y,t)\in\left\{(t,x,y,z)\in\L^4\setm t-z=0\right\}.\]
From now on, we use this identification without explicit mention.
We denote by the same bracket $\inner{\cdot}{\cdot}$ the canonical Lorentzian inner product on $\L^4$.

Let $\Sigma^2$ be a Riemann surface.
An immersion $f\col\Sigma^2\to\I^3$ is said to be \textit{space-like} if the first fundamental form $ds^2:=f^*\inner{\cdot}{\cdot}$ is a positive definite metric on $\Sigma^2$.
Assume that a space-like immersion $f\col\Sigma^2\to\I^3$ is conformal with respect to $ds^2$.
We define the constant vector $\mathfrak{p}\in\I^3$ as $\mathfrak{p}:=(1,0,0)$, and let us consider the subset
\[\mathcal{P}:=\{\vb*{v}\in\L^4\setm\inner{\vb*{v}}{\vb*{v}}=0,\;\inner{\vb*{v}}{\mathfrak{p}}=1\}.\]
The subset $\mathcal{P}$ can be regarded as the range of the Gauss map of a space-like surface in $\I^3$.
More precisely, given a space-like immersion $f\col\Sigma^2\to\I^3$, then there exists a unique map $\nu\col\Sigma^2\to \mathcal{P}$ satisfying $\inner{\nu}{df}=0$.
The map $\nu$ is called the \textit{light-like Gauss map} of $f$.
We define the \textit{second fundamental form} $h$ of $f$ as \[h:=-\inner{df}{d\nu}.\]
Let $\nabla$ be the Levi-Civita connection of $ds^2$, and let $K_{ds^2}$ be the Gaussian curvature of $ds^2$.
We denote by $\mathfrak{X}(\Sigma^2)$ the set of tangent vector fields on $\Sigma^2$.
Then, the Gauss and Codazzi equations are given by
\begin{align}
  K_{ds^2}&\equiv 0,\label{eq_Gauss}\\
 (\nabla_Xh)(Y,Z)&=(\nabla_Yh)(X,Z)\qquad(X,Y,Z\in\mathfrak{X}(\Sigma^2)),\label{eq_Codazzi}
\end{align}
respectively. \eqref{eq_Gauss} implies that any surface in $\I^3$ is intrinsically flat.

We call a map $\tau\col\I^3\to\I^3$ defined by the following expression an \textit{isometry}.
\begin{align}\label{eq_isometry}
  \tau\begin{pmatrix}
    t\\
    x\\
    y
  \end{pmatrix}:=
  \begin{pmatrix}
    \pm 1 & a & b\\
    \multicolumn{1}{c}{
     \begin{array}{@{} c @{}}
       0 \\
       0 
     \end{array}
    }
    &
    \multicolumn{2}{c}{
     \begin{array}{@{} c @{}}
       \vcenter{\hbox{\scalebox{1.3}{$T$}}}
     \end{array}
    }
  \end{pmatrix}
  \begin{pmatrix}
    t\\
    x\\
    y
  \end{pmatrix}
  +
  \begin{pmatrix}
    t_0\\
    x_0\\
    y_0
  \end{pmatrix}
  \qquad\left(
  T\in\textrm{O}(2),\;a,b,t_0,x_0,y_0\in\R
  \right).
\end{align}
The set of isometries of $\I^3$ is a Lie group with four connected components.
As in the case of surfaces in the Euclidean 3-space $\E^3$, the fundamental theorem of surface theory in $\I^3$ is also known (see Sachs~\cite[Theorem~8.8]{Sachs1990} and Sato~\cite[Theorem~5.4]{Sato2021} for details).

Let $f\col\Sigma^2\to\I^3$ be a space-like immersion.
The \textit{mean curvature function} $H$ of $f$ is defined as half of the trace of the second fundamental form $h$ with respect to $ds^2$.
We say that $f$ is a \textit{zero mean curvature surface} in $\I^3$ if $H$ is identically zero.
As stated in Introduction, the Weierstrass-type representation formula \eqref{eq_intro_Weierstrass_representation_formula} for zero mean curvature surfaces in $\I^3$ is known.
The pair $(g,\omega)$ is called the \textit{Weierstrass data} of $f$.
Some typical examples of zero mean curvature surfaces in $\I^3$ are as follows:

\begin{example}[Plane]
  Let $\Sigma^2:=\C$ and $(g,\omega):=(0,dz)$.
  Then, the corresponding zero mean curvature surface $f$ is given by
  \begin{align}\label{eq_plane}
    f(u+iv)=(0,u,v)\qquad(u,v\in\R).
  \end{align}
  $f$ is a totally geodesic surface in $\I^3$, that is, the second fundamental form $h$ of $f$ is identically zero.
  This surface is called a (space-like) \textit{plane} in $\I^3$.
\end{example}
\begin{example}[Enneper paraboloid]\label{exam_Enneper_paraboloid}
  Let $\Sigma^2:=\C$ and $(g,\omega):=(z,dz).$
  Then, the corresponding zero mean curvature surface $f$ is given by
  \begin{align}\label{eq_Enneper_paraboloid}
    f(u+iv)=\left(\frac{1}{2}(u^2-v^2),u,v\right)\qquad(u,v\in\R).
  \end{align}
  This surface is called an \textit{Enneper paraboloid} in $\I^3$.
\end{example}
\begin{example}[Catenoid]\label{exam_catenoid}
  Let $\Sigma^2:=\C\setminus\{0\}$ and $(g,\omega):=(1/z,dz).$
  Then, the corresponding zero mean curvature surface $f$ is given by
  \begin{align}\label{eq_catenoid}
    f(u+iv)=\left(\frac{1}{2}\log(u^2+v^2),u,v\right)\qquad(u,v\in\R,\;(u,v)\neq(0,0)).
  \end{align}
  This surface is called a \textit{catenoid} in $\I^3$ (see Figure~\ref{fig_catenoid_sing1} (a)).
  The image of $f$ is invariant under revolution around the $t$-axis.
  Remark that the first fundamental form $ds^2$ is complete at $z=\infty$ but not complete at $z=0$.
\end{example}

\section{ZMC-faces in the isotropic 3-space}\label{sec_2}

In this section, we introduce a class of zero mean curvature surfaces with singularities as \textit{ZMC-faces} in $\I^3$, and discuss the Weierstrass-type representation formula.
Throughout this section, we assume that $\Sigma^2$ denotes a Riemann surface and $\pi\col\tilde\Sigma^2\to\Sigma^2$ is the universal cover of $\Sigma^2$.

\subsection{The definition of ZMC-faces}

To begin with, we give the definition of ZMC-faces.

\begin{definition}\label{def_ZMC_face}
  A holomorphic immersion $F=(F^0,F^1,F^2)\col\tilde\Sigma^2\to\C^3$ is said to be \textit{$(0++)$-type null} if
  \begin{align}\label{eq_isotropic_null}
    (dF^1)^2+(dF^2)^2=0
  \end{align}
  holds and $F^1$ is non-constant (equivalently $F^2$ is non-constant).

  We say that a map $f\col\Sigma^2\to\I^3$ is a \textit{ZMC-face} if there exists a holomorphic $(0++)$-type null immersion $F\col\tilde\Sigma^2\to\C^3$ satisfying $f\circ\pi=\Re(F)$.
  We call $F$ the \textit{holomorphic lift} of $f$.
  A point $p\in\Sigma^2$ is said to be \textit{regular} if the induced metric $ds^2:=f^*\inner{\cdot}{\cdot}$ is positive definite at $p$.
  Moreover, $p$ is said to be \textit{singular} if $p$ is not regular.
\end{definition}

\begin{remark}
  In general, a holomorphic immersion $F=(F^0,F^1,F^2)\col\tilde\Sigma^2\to\C^3$ satisfying
  \begin{align}\label{eq_Euclidean_null}
    (dF^0)^2+(dF^1)^2+(dF^2)^2=0
  \end{align}
  is said to be \textit{Euclidean null}. On the other hand, $F$ is said to be \textit{Lorentzian null} if $F$ satisfies
  \begin{align}\label{eq_Lorentzian_null}
    -(dF^0)^2+(dF^1)^2+(dF^2)^2=0.
  \end{align}
  Therefore, \eqref{eq_isotropic_null} is an expression intermediate between \eqref{eq_Euclidean_null} and \eqref{eq_Lorentzian_null}.
\end{remark}

\begin{remark}
  We also impose the condition that $F^1$ and $F^2$ are non-constant in the definition of $(0++)$-type null, since otherwise the image of $f:=\Re(F)$ is contained in a line parallel to the $t$-axis, which is completely degenerate.
\end{remark}

The left-hand side of \eqref{eq_isotropic_null} can be factorized as
\[(dF^1)^2+(dF^2)^2=(dF^1-i\,dF^2)(dF^1+i\,dF^2).\]
Therefore, \eqref{eq_isotropic_null} is equivalent to
\begin{align}\label{eq_isotropic_null_general}
  dF^1-i\,dF^2=0\quad\textrm{or}\quad dF^1+i\,dF^2=0.
\end{align}
In the latter case, by multiplying $F^2$ by $-1$, we can reduce to it to the former case.
For $f:=\Re(F)$, this operation corresponds to the reflection in the $tx$-plane.
To simplify the subsequent discussion, unless explicitly stated otherwise, \underline{we shall only consider} a ZMC-face $f$ whose holomorphic lift $F$ satisfies \[dF^1-i\,dF^2=0.\]

The Weierstrass-type representation formula for ZMC-faces can be written as follows:

\begin{proposition}[Weierstrass-type representation formula for ZMC-faces]\label{prop_representation_formula_ZMC_face}
  Let $g$ and $\omega$ be a meromorphic function and a holomorphic 1-form on $\Sigma^2$, respectively.
  Suppose that the pair $(g,\omega)$ satisfies the following two conditions.
  \begin{enumerate}
    \renewcommand{\theenumi}{\ifcase\value{enumi}\or{\rm (C)}\or{\rm (P)}\fi}
    \renewcommand{\labelenumi}{\theenumi}
    \item\label{item_compatibility_condition}
      $g\omega$ is holomorphic, while $\omega$ and $g\omega$ have no common zeros.
    \item\label{item_period_condtion} For any closed curve $\gamma$ on $\Sigma^2$,
      \begin{align}\label{eq_period_condition_ZMC_face}
        \Re\oint_\gamma g\omega=0\quad\textrm{and}\quad\oint_\gamma\omega=0
      \end{align}
      hold.
  \end{enumerate}
  Then, the map $f\col\Sigma^2\to\I^3$ defined by
  \begin{align}\label{eq_Weierstrass_representation_formula_for_ZMC_face_in_thm}
    f:=\Re\int(g,1,-i)\omega
  \end{align}
  is a ZMC-face.
   
  Conversely, any ZMC-face can be represented in this way.
\end{proposition}

The detailed proof of Proposition~\ref{prop_representation_formula_ZMC_face} is given in Subsection~\ref{subsec_2_2}.
We call the conditions \ref{item_compatibility_condition} and \ref{item_period_condtion} in Proposition~\ref{prop_representation_formula_ZMC_face} the \textit{compatibility condition} and the \textit{period condition}, respectively.
We also call the pair $(g,\omega)$ the \textit{Weierstrass data} of a ZMC-face $f$.
In particular, $g$ is called the \textit{Gauss map} of $f$ (cf. Remark~\ref{rem_geomety_of_gauss_map}).

\begin{remark}\label{rem_representation_formula}
  Several variants of the Weierstrass-type representation formula for zero mean curvature surfaces in $\I^3$ are known in \cite{CLLY2024, daSilva2021, Pember2020, Sato2021, SY2021, Strubecker1942}.
  In particular, Sato~\cite{Sato2021} considers zero mean curvature surfaces with singularities, including branch points, namely singular points of corank 2.
  In addition, da Silva~\cite{daSilva2021} and Seo-Yang~\cite{SY2021} discuss the Bj\"orling representation.
\end{remark}
\begin{remark}
  Let $f$ and $\tilde f$ be ZMC-faces obtained from Weierstrass data $(g,\omega)$ and $(\tilde g,\tilde\omega)$, respectively.
  Since an isometry of $\I^3$ can be expressed in \eqref{eq_isometry}, $f$ and $\tilde f$ are coincide up to an isometry if and only if there exist $\theta\in[0,2\pi),\;c\in\C$ and $\epsilon_1,\epsilon_2\in\{-1,1\}$ satisfying
  \[\tilde g=\epsilon_1e^{i\theta}g+c,\qquad\tilde\omega=\epsilon_2e^{-i\theta}\omega.\]
  Note that multiplying $g$ and $\omega$ by a positive real constant $s$ corresponds to a dilation with respect to the $t$-axis (namely, $(t,x,y)\mapsto (st,x,y)$) and a homothetic transformation (namely, $(t,x,y)\mapsto (st,sx,sy)$), respectively.
\end{remark}

Before we provide the proof of Proposition~\ref{prop_representation_formula_ZMC_face}, we calculate the first and second fundamental forms of $f$ using the Weierstrass data $(g,\omega)$.

\begin{proposition}\label{prop_first_second_fundamental_form_ZMC_face}
  Let $f\col\Sigma^2\to\I^3$ be a ZMC-face and $(g,\omega)$ be the Weierstrass data of $f$.
  Then, the first fundamental form $ds^2$ and the second fundamental form $h$ of $f$ are given by
  \begin{align}\label{eq_first_second_fundamental_form_ZMC_face}
    ds^2=|\omega|^2,\qquad h=\frac{1}{2}\omega dg+\frac{1}{2}\bar\omega d\bar g.
  \end{align}
  The set of singular points of $f$ coincides with the zeros of $\omega$, equivalently, with the poles of $g$.
  Moreover, $f$ has a zero mean curvature on the set of regular points.
\end{proposition}
\begin{proof}
  Since we have
  \begin{align}\label{eq_df}
    df=\frac{1}{2}(g,1,-i)\omega+\frac{1}{2}(\bar g,1,i)\bar\omega
  \end{align}
  by \eqref{eq_Weierstrass_representation_formula_for_ZMC_face_in_thm}, the first fundamental form can be calculated as
  \[ds^2=\inner{df}{df}=|\omega|^2.\]
  It follows that $p\in\Sigma^2$ is a singular point of $f$ if and only if $p$ is a zero of $\omega$.
  By the compatibility condition of $(g,\omega)$, the zeros of $\omega$ coincide with the poles of $g$.
  
  Let $z$ be a complex coordinate of $\Sigma^2$, then we have
  \[f_{zz}=\frac{1}{2}(g_z\hat\omega+g\hat\omega_z,\hat\omega_z,-i\hat\omega_z),\qquad f_{z\bar z}=0,\]
  where $\hat\omega:=\omega/dz$.
  Moreover, one can check that
  \begin{align}\label{eq_lightlike_Gauss_map_representation}
    \nu=\left(-\frac{1}{2}|g|^2-\frac{1}{2},-\Re g,\Im g,-\frac{1}{2}|g|^2+\frac{1}{2}\right)
  \end{align}
  is exactly the light-like Gauss map of $f$ (cf. \cite[Theorem 3.9]{CLLY2024}).
  Therefore, the second fundamental form $h$ can be calculated as
  \[h=\inner{f_{zz}}{\nu}dz^2+2\inner{f_{z\bar z}}{\nu}dzd\bar z+\inner{f_{\bar z\bar z}}{\nu}d\bar z^2=\frac{1}{2}\omega dg+\frac{1}{2}\bar\omega d\bar g.\]
  Moreover, it follows that $f$ has a zero mean curvature on the set of regular points.
\end{proof}

\begin{remark}\label{rem_geomety_of_gauss_map}
  Let us consider the bijection $\Pi\col\mathcal{P}\to\C$ defined by
  \[\Pi(t,x,y,z):=\frac{x-iy}{t-z}=-x+iy.\]
  Then, by \eqref{eq_lightlike_Gauss_map_representation} we have $\Pi\circ\nu=g$ on the set of regular point of $f$.
  This is the reason that $g$ is called the Gauss map of $f$.
\end{remark}

\subsection{Proof of Proposition~\ref{prop_representation_formula_ZMC_face}}\label{subsec_2_2}

To show Proposition~\ref{prop_representation_formula_ZMC_face}, we prepare two lemmas.

\begin{lemma}\label{lem_compatibility_condtion}
  The following statements are equivalent.
  \begin{enumerate}[{\rm (i)}]
    \item $(g,\omega)$ satisfies the compatibility condition.
    \item $g\omega$ is holomorphic, and if we set a holomorphic map $F\col\tilde\Sigma^2\to\C^3$ as
      \begin{align}\label{eq_lift_by_Weierstrass_data}
        F:=\int(g,1,-i)\omega,
      \end{align}
      then $F$ is an immersion.
  \end{enumerate}
  Moreover, if one of these two conditions holds, the holomorphic immersion $F\col\tilde\Sigma^2\to\C^3$ defined by \eqref{eq_lift_by_Weierstrass_data} is $(0++)$-type null.
\end{lemma}
\begin{proof}
  In both statements, it is assumed that $g\omega$ is holomorphic, so we assume this as well.
  Let $F$ be the holomorphic map defined by \eqref{eq_lift_by_Weierstrass_data}, and write $F=(F^0,F^1,F^2)$.
  Since we have
  \[|dF^0|^2+|dF^1|^2+|dF^2|^2=(2+|g|^2)|\omega|^2,\]
  it follows that $F$ is an immersion if and only if $(g,\omega)$ satisfies the compatibility condition.
  
  Assume that $(g,\omega)$ satisfies the compatibility condition, let us check that $F$ is $(0++)$-type null.
  It is clear that $F$ satisfies \eqref{eq_isotropic_null}.
  Since $\omega$ does not vanish everywhere by the compatibility condition, $F^1$ is non-constant.
  Therefore, $F$ satisfies the definition of $(0++)$-type null.
\end{proof}

\begin{lemma}\label{lem_period_condtion}
  Suppose that $(g,\omega)$ satisfies the compatibility condition.
  Then, the following statements are equivalent.
  \begin{enumerate}[{\rm (i)}]
    \item $(g,\omega)$ satisfies the period condition.
    \item\label{item_period_condtion_2} If we define $F\col\tilde\Sigma^2\to\C^3$ by \eqref{eq_lift_by_Weierstrass_data},
      then $\Re(F)$ is single-valued on $\Sigma^2$.
  \end{enumerate}
  Moreover, if one of these two conditions holds, the map $f\col\Sigma^2\to\I^3$ defined by $f:=\Re(F)$ is a ZMC-face.
\end{lemma}
\begin{proof}
  The statement in \eqref{item_period_condtion_2} is equivalent to
  \begin{align}\label{eq_lem_period}
    \Re\oint_\gamma(g,1,-i)\omega=(0,0,0)
  \end{align}
  for any close curve $\gamma$ on $\Sigma^2$.
  Since $\Re\oint_\gamma-i\omega=\Im\oint_\gamma\omega$ holds, it is clear that \eqref{eq_lem_period} is equivalent to the period condition of $(g,\omega)$.
  Moreover, since $F$ is a holomorphic $(0++)$-type null immersion by Lemma~\ref{lem_compatibility_condtion}, the map $f:=\Re(F)$ is a ZMC-face if $(g,\omega)$ satisfies the period condition in addition to the compatibility condition.
\end{proof}

\begin{proof}[Proof of Proposition~\ref{prop_representation_formula_ZMC_face}]
  By Lemma~\ref{lem_period_condtion}, it follows that $f\col\Sigma^2\to\I^3$ defined by \eqref{eq_Weierstrass_representation_formula_for_ZMC_face_in_thm} is a ZMC-face if $(g,\omega)$ satisfies the compatibility and period conditions.
  
  Conversely, let $f\col\Sigma^2\to\I^3$ be a ZMC-face and $F\col\tilde\Sigma^2\to\I^3$ the holomorphic lift of $f$.
  We define a meromorphic function $g$ and a holomorphic 1-form $\omega$ on $\Sigma^2$ as
  \begin{align}\label{eq_def_Weierstrass_data}
    g:=\frac{2\,\partial f^0}{\partial f^1+i\,\partial f^2},\qquad\omega:=\partial f^1+i\,\partial f^2,
  \end{align}
  where we denote $f=(f^0,f^1,f^2)$.
  By $f\circ\pi=\Re(F)$, we have
  \begin{align}\label{eq_partial_f_and_dF}
    \partial(f\circ\pi)=\frac{1}{2}dF.
  \end{align}
  Since $F$ is $(0++)$-type null, we have $\partial f^1-i\,\partial f^2=0$, and $\omega$ can be written as $\omega=2\partial f^1$.
  Moreover, since $F^1$ is non-constant, $f^1$ is also non-constant.
  This implies that $\omega$ does not vanish identically and $g\omega$ is holomorphic.
  By \eqref{eq_def_Weierstrass_data} and \eqref{eq_partial_f_and_dF}, we have
  \[dF=(g,1,-i)\omega.\]
  Therefore, by Lemmas~\ref{lem_compatibility_condtion} and \ref{lem_period_condtion}, the pair $(g,\omega)$ satisfies the compatibility and period conditions.
\end{proof}

\subsection{Singularities of ZMC-faces}

In this subsection, we discuss singular points of ZMC-faces.
The following proposition states that a singular point of a ZMC-face agrees with that of the mapping.

\begin{proposition}\label{prop_equivalent_singularity}
  Let $f\col\Sigma^2\to\I^3$ be a ZMC-face, and let $p\in\Sigma^2$ be a point.
  Then, $p$ is singular if and only if $f$ is not an immersion at $p$.
\end{proposition}
\begin{proof}
  Let $(g,\omega)$ be the Weierstrass data of $f$.
  We take a complex coordinate $z=u+iv$ of $\Sigma^2$ centered at $p$, and we set $\hat\omega:=\omega/dz$.
  We denote by $\times_E$ the vector product of $\E^3$.
  Then, we have
  \begin{align}\label{eq_Euclidean_cross_product}
    N:=f_u\times_E f_v=-2if_z\times_E f_{\bar z}=|\hat\omega|^2\bigl(1,-\Re(g),\Im(g)\bigr).
  \end{align}
  Since $g\omega$ is holomorphic, the points at which $N=0$ are precisely the zeros of $\omega$.
  Therefore, by Proposition~\ref{prop_first_second_fundamental_form_ZMC_face} the assertion follows.
\end{proof}

Here, we recall some basic notations on singularities of smooth maps.
We denote by $\mathbb{S}^2$ the unit sphere centered at the origin in $\R^3$.
A smooth map $f\col\Sigma^2\to\R^3$ is called a \textit{frontal} if for each $p\in\Sigma^2$ there exist a neighborhood $U$ of $p$ and a smooth map $\vb*{n}\col U\to\mathbb{S}^2$ satisfying
\begin{align}\label{eq_Euclidean_normal_vector}
  \inner{\vb*{n}}{df}_E=0,
\end{align}
where $\inner{\cdot}{\cdot}_E$ is the canonical Euclidean inner product on $\R^3$.
For a maxface, it is always a frontal (cf. \cite[Lemma 3.3]{UY2006}); however, we can show that a ZMC-face with at least one singular point is not frontal as follows:

\begin{proposition}
  A ZMC-face $f\col\Sigma^2\to\I^3$ is not a frontal if there exists a singular point of $f$.
\end{proposition}
\begin{proof}
  We use the notation introduced in the proof of Proposition~\ref{prop_equivalent_singularity}.
  Then, on the set of regular points of $f$, a map $\vb*{n}$ satisfying \eqref{eq_Euclidean_normal_vector} can be calculated as
  \[\vb*{n}:=\frac{N}{\sqrt{\inner{N}{N}_E}}=\frac{1}{\sqrt{|g|^2+1}}\bigl(1,-\Re(g),\Im(g)\bigr).\]
  Suppose that $p$ is a singular point of $f$.
  Since $g$ has a pole at a singular point $p$ by Proposition~\ref{prop_first_second_fundamental_form_ZMC_face}, we can take $g=z^{-m}$ by changing a complex coordinate $z$, where $m$ is a positive integer.
  If we represent $z$ in the polar coordinate as $z=re^{i\theta}$, then $\vb*{n}$ is rewritten as
  \[\vb*{n}=\frac{1}{\sqrt{1+r^{2m}}}\bigl(r^m,-\cos(m\theta),-\sin(m\theta)\bigr).\]
  Since the limit of $\vb*{n}$ as $r\to 0$ depends on the value of $\theta$, it follows that $f$ is not a frontal.
\end{proof}

\begin{proposition}
  Singular points of a ZMC-face are isolated and have corank 1.
\end{proposition}
\begin{proof}
  Since a singular point corresponds to a zero of $\omega$, singular points of a ZMC-face are isolated.
  By the compatibility condition, $g\omega$ does not have zero at a singular point $p$.
  Therefore, $df$ does not vanish at $p$ by \eqref{eq_df}, and this implies that the corank at $p$ is 1.
\end{proof}

\begin{example}\label{exam_singular_points}
  Let $\Sigma^2:=\C$ and $(g,\omega):=(z^{-m},z^mdz)$, where $m$ is a positive integer.
  We can check that $(g,\omega)$ satisfies the compatibility and period conditions.
  The corresponding ZMC-face $f$ has a singular point at $z=0$ (see Figure~\ref{fig_singular_points}).
  Note that the cases $m=1,2$ are already given in Sato~\cite[Figure 2]{Sato2021} as examples of zero mean curvature surfaces with singularities.
\end{example}

\begin{figure}[htbp]
  \centering
  \begin{tabular}{c@{\hspace{1cm}}c}
    \includegraphics[width=3cm]{sing1.png} &
    \includegraphics[width=3.2cm]{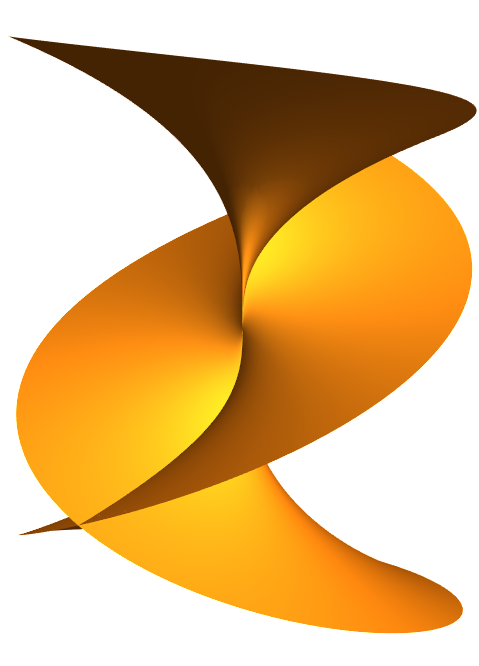}
  \end{tabular}
  \caption{Example~\ref{exam_singular_points} (Left: $m=1$, Right: $m=2$).}
  \label{fig_singular_points}
\end{figure}

For a smooth map $f\col\Sigma^2\to\R^3$ having a singular point at $p\in\Sigma^2$, we say that $f$ has a \textit{cross cap} (or a \textit{Whitney's umbrella}) at $p$ if there exist a smooth coordinate $(u,v)$ of $\Sigma^2$ centered at $p$ and a diffeomorphism $\Phi$ from a neighborhood of $f(0,0)$ to a neighborhood of the origin in $\R^3$ such that
\[\Phi\circ f(u,v)=(u,uv,v^2).\]
There is a well-known criterion by Whitney \cite{Whitney1944} to determine whether a singular point is a cross cap.
Using this criterion, one can check that in Example~\ref{exam_singular_points} with $m=1$, the ZMC-face $f$ has a cross cap at $z=0$.
Moreover, we obtain a simple criterion for a singular point of a ZMC-face to be a cross cap as follows:

\begin{proposition}\label{prop_criterion_cross_cap}
  Let $f\col\Sigma^2\to\I^3$ be a ZMC-face, and $p\in\Sigma^2$ be a singular point of $f$.
  Then, $f$ has a cross cap at $p$ if and only if $\ord_p(\omega)=1$ holds.
\end{proposition}

The proof of Proposition~\ref{prop_criterion_cross_cap} is given in Appendix~\ref{app_A}.

\section{Dual surfaces of ZMC-faces}\label{sec_3}

It is known that space-like surfaces in $\I^3$ have a certain duality (cf. Strubecker~\cite{Strubecker1978} and Pottmann-Liu~\cite{PL2007}).
In particular, the ``dual surface'' of a zero mean curvature surface in $\I^3$ has also a zero mean curvature (on the set of regular points).
As a generalization of this duality, we introduce the ``dual map'' $f_*$ of a given ZMC-face $f$ and discuss the relationship between their Weierstrass data.

Let $f=(f^0,f^1,f^2)\col\Sigma^2\to\I^3$ be a ZMC-face and $(g,\omega)$ be the Weierstrass data of $f$.
We denote by $S_f(\subseteq\Sigma^2)$ the set of singular points of $f$.
We define a holomorphic function $g_*$ on $\Sigma^2$ by
\begin{align}\label{eq_def_gs}
  g_*:=f^1+if^2=\int\omega.
\end{align}
Moreover, we define a meromorphic 1-form $\omega_*$ on $\Sigma^2$ by
\begin{align}\label{eq_def_omegas}
  \omega_*:=dg,
\end{align}
which is holomorphic on $\Sigma^2_*:=\Sigma^2\setminus S_f$.
We also denote by $\pi_*\col\tilde\Sigma^2_*\to\Sigma^2_*$ the universal cover of $\Sigma^2_*$.
Let us consider a holomorphic map $F_*\col\tilde\Sigma^2_*\to\C^3$ defined as
\[F_*:=\int(g_*,1,i)\omega_*.\]
Since $g_*$ is holomorphic, $F_*$ is $(0++)$-type null immersion except at the zeros of $\omega_*$.
Note that $F_*$ satisfies the latter case of \eqref{eq_isotropic_null_general}.
Using integration by parts, we obtain
\begin{align*}
  \oint_\gamma F_*=\left(-\oint_\gamma g\omega,0,0\right),
\end{align*}
for any closed curve $\gamma$ on $\Sigma^2_*$.
Therefore, the period condition for the pair $(g_*,\omega_*)$ follows from that for $(g,\omega)$.
This implies that the map $f_*\col\Sigma^2_*\to\I^3$ defined as $f_*:=\Re(F_*)$ is single-valued.
Moreover, $f_*$ is a zero mean curvature surface on the regular points.

\begin{definition}\label{def_dual_surface}
  For a given ZMC-face $f\col\Sigma^2\to\I^3$, the map $f_*$ defined as above is called the \textit{dual map} of $f$.
  We call the pair $(g_*,\omega_*)$ defined by \eqref{eq_def_gs} and \eqref{eq_def_omegas} the \textit{dual Weierstrass data} of $f$.
  In particular, $g_*$ is called the \textit{dual Gauss map} of $f$.
\end{definition}

Combining \eqref{eq_def_gs} and \eqref{eq_def_omegas}, the relationship between the Weierstrass data $(g,\omega)$ and the dual Weierstrass data $(g_*,\omega_*)$ is expressed as follows:
\begin{align}\label{eq_relation_dual_Weierstrass_data}
  \omega=dg_*,\qquad\omega_*=dg.
\end{align}

\begin{remark}
  The dual map may not be a ZMC-face, in general.
  For example, let us consider the Weierstrass data $(g,\omega)=(z^2,dz)$ on the entire plane $\C$.
  Then, the dual Weierstrass data is given by $(g_*,\omega_*)=(z,2zdz)$, and the map $F_*$ is not an immersion at $z=0$.
  Therefore, the dual map $f_*$ has a branch point at $z=0$ and is not a ZMC-face.
\end{remark}

\begin{example}
  Let $\phi(x,y)$ be a harmonic function, and let us consider the case that the ZMC-face $f$ given by $f(x+iy):=(\phi(x,y),x,y)$ as the graph of $\phi$.
  Since the Weierstrass data of $f$ is $(g,\omega)=(2\phi_z,dz)$, where $z=x+iy$, the dual Weierstrass data of $f$ can be written as $(g_*,\omega_*)=\left(z,2\phi_{zz}dz\right)$.
  Therefore, the dual map $f_*$ can be calculated as
  \[f_*=\Re\int(z,1,i)2\phi_{zz}dz=\bigl(x\phi_x+y\phi_y-\phi,\phi_x,\phi_y\bigr).\]
  This expression is the same as that in \cite{PL2007, Strubecker1978}.
\end{example}

\section{Weakly complete and finite-type ZMC-faces}\label{sec_4}

In this section, we discuss the weak completeness and the finiteness of the total curvature for a ZMC-face.
We define these property using the ``lift-metric'', which is the same way as in \cite{UY2006}.

\begin{definition}[cf. {\cite[Definition~2.7]{UY2006}}]
  Let $f\col\Sigma^2\to\I^3$ be a ZMC-face and $F=(F^0,F^1,F^2)$ the holomorphic lift of $f$, and let $(g,\omega)$ be the Weierstrass data of $f$.
  We define a Riemannian metric $ds^2_{\!\#}$ on $\Sigma^2$ by
  \begin{align}\label{eq_lift_metric}
    ds^2_{\!\#}:=(2+|g|^2)|\omega|^2=|dF^0|^2+|dF^1|^2+|dF^2|^2,
  \end{align}
  and we call it the \textit{lift-metric} with respect to $f$.
\end{definition}

\begin{definition}\label{def_weakly_complete_and_finite_type}
  A ZMC-face $f\col\Sigma^2\to\I^3$ is said to be \textit{weakly complete} (resp. \textit{finite-type}) if the lift-metric $ds^2_{\!\#}$ is a complete Riemannian metric (resp. a Riemannian metric with finite absolute total Gaussian curvature) on $\Sigma^2$.
\end{definition}

\begin{remark}
  From \eqref{eq_lift_metric} together with Kawakami~\cite[Theorem 2.1]{Kawakami2013}, it follows that the maximum number of omitted values of the Gauss map $g$ of a weakly complete ZMC-face, excluding a plane, is 3.
  The number of omitted values of the Gauss map $g$ under the additional assumption of finite-type, as well as weak completeness, will be discussed in Section~\ref{sec_7}.
\end{remark}

Since we have $ds^2\le ds^2_{\!\#}$ by \eqref{eq_lift_metric}, completeness for ZMC-faces implies weak completeness.
However, a catenoid (cf. Example~\ref{exam_catenoid}) is weakly complete but not complete.

\begin{proposition}\label{prop_Huber_Osserman}
  Let $f\col\Sigma^2\to\I^3$ be a weakly complete and finite-type ZMC-face.
  Then, there exist a compact Riemann surface $\bar\Sigma^2$ and distinct points $p_1,\dots,p_n\in\bar\Sigma^2$ such that $\Sigma^2$ is biholomorphic to $\bar\Sigma^2\setminus\{p_1,\dots,p_n\}$.
  Moreover, the Weierstrass data of $f$ can be extended meromorphically on $\bar\Sigma^2$.
\end{proposition}
\begin{proof}
  We denote by $K_{\!\#}$ the Gaussian curvature of the lift-metric $ds^2_{\!\#}$.
  Then, we have
  \begin{align}\label{eq_relation_lift_Gauss_curvature_and_Fubini_Study_metric}
    (-K_{\!\#})ds^2_{\!\#}=\frac{4|dg|^2}{(2+|g|^2)^2}=\frac{1}{2}\frac{4|d(g/\sqrt{2})|^2}{(1+|g/\sqrt{2}|^2)^2}=\frac{1}{2}(\mu\circ g)^*ds^2_{\mathrm{FS}},
  \end{align}
  where $\mu$ is the map $\C\cup\{\infty\}\ni z\mapsto z/\sqrt{2}\in\C\cup\{\infty\}$ and $ds^2_{\mathrm{FS}}$ is the Fubini-Study metric of the Riemann sphere.
  In particular, $K_{\#}$ has non-positive value.
  Therefore, it follows from Huber's theorem \cite{Huber1958} (cf. \cite[Appendix~A]{UY2006}) that $\Sigma^2$ is biholomorphic to $\bar\Sigma^2\setminus\{p_1,\dots,p_n\}$.
  Using \eqref{eq_relation_lift_Gauss_curvature_and_Fubini_Study_metric} and the Great Picard theorem, we can show that $g$ has at most pole at $p_j$ for each $j=1,\dots,n$.
  In the case that $g$ does not have a pole at $p_j$, it follows from weak completeness of $f$ that $|\omega|^2$ is complete at $p_j$.
  By using \cite[Lemma~9.6]{Osserman1986}, $\omega$ has a pole at $p_j$.
  Let us consider the case that $g$ has a pole at $p_j$.
  Since the lift-metric is written as
  \[ds^2_{\!\#}=\left(\frac{2}{|g|^2}+1\right)|g\omega|^2,\]
  $|g\omega|^2$ is complete at $p_j$.
  Therefore, by the same reasoning, $\omega$ has at most pole at $p_j$.
\end{proof}

Proposition~\ref{prop_finite_type_entire_ZMC} in Introduction can be shown as follows:

\begin{proof}[Proof of Proposition~\ref{prop_finite_type_entire_ZMC}]
  We set $z=x+iy$ and denote by $(g,\omega)$ the Weierstrass data of $f(x+iy)=(\phi(x,y),x,y)$.
  Since $(g,\omega)=(2\phi_z,dz)$ holds, $f$ is finite-type if and only if $\phi_z$ is a rational function on the Riemann sphere.
  Since $g$ has no poles on $\C$, this is equivalent that $\phi_z$ is a polynomial.
  Therefore, $f$ is finite-type if and only if $\phi$ is a harmonic polynomial.
\end{proof}

\begin{remark}
  For a weakly complete and finite-type ZMC-face $f$, by Proposition~\ref{prop_Huber_Osserman} and \eqref{eq_relation_lift_Gauss_curvature_and_Fubini_Study_metric} we have
  \[\deg(g)=\frac{1}{2\pi}\int_{\Sigma^2}(-K_{\!\#})dA_{\#},\]
  where $dA_{\#}$ is the volume form with respect to the lift-metric $ds^2_{\!\#}$.
  Therefore, \eqref{eq_third_Osserman_type_inequality_intro} can be rewritten as
  \[\frac{1}{2\pi}\int_{\Sigma^2}(-K_{\!\#})dA_{\#}\ge n+k-\chi(\bar\Sigma^2).\]
\end{remark}

We call each point $p_j$ in Proposition~\ref{prop_Huber_Osserman} an \textit{end} of $f$.
We say that an end $p$ is \textit{embedded} if there exists a neighborhood $U\subseteq\bar\Sigma^2$ of $p$ such that the restriction of $f$ to $U\setminus\{p\}$ is an embedding.

The following lemma will be used in the subsequent discussion.

\begin{lemma}\label{lem_residue_omega}
  Let $f\col\Sigma^2\to\I^3$ be a weakly complete and finite-type ZMC-face, and let $(g,\omega)$ be the Weierstrass data of $f$.
  Then, the residue of $\omega$ at an end is zero.
\end{lemma}
\begin{proof}
  This follows immediately from the period condition \eqref{eq_period_condition_ZMC_face}.
\end{proof}

As stated in Example~\ref{exam_catenoid}, a catenoid has two different ends: one end is complete, while the other is not complete but weakly complete.
To distinguish between them, we introduce the following definition.

\begin{definition}\label{def_expanding_shrinking}
  Let $f\col\Sigma^2\to\I^3$ be weakly complete and finite-type ZMC-face.
  An end $p$ of $f$ is said to be \textit{expanding} (resp. \textit{shrinking}) if the first fundamental form $ds^2$ is complete (resp. not complete) at $p$.
\end{definition}

\begin{proposition}\label{prop_equivalent_expanding_shrinking}
  An end $p$ is expanding (resp. shrinking) if and only if $\ord_p(\omega)\le -2$ {\rm (}resp. $\ord_p(\omega)\ge 0${\rm )} holds, where $\ord_p(\omega)$ denotes the order of $\omega$ at $p$.
\end{proposition}
\begin{proof}
  By Proposition~\ref{prop_first_second_fundamental_form_ZMC_face}, the first fundamental form $ds^2$ is complete at an end $p$ if and only if $\omega$ has a pole at $p$.
  Since $\ord_p(\omega)\neq -1$ by Lemma~\ref{lem_residue_omega}, the assertion follows.
\end{proof}

As is well-known in the theory of minimal surfaces in the Euclidean 3-space $\E^3$, Jorge-Meeks~\cite{JM1983} and Schoen~\cite{Schoen1983} showed that complete end which has the finite total curvature is embedded if and only if it is asymptotic to an end of either a plane or a catenoid.
By contrast, the condition for an end of a ZMC-face in $\I^3$ to be embedded is quite relaxed as follows:

\begin{proposition}\label{prop_embedded_end}
  Let $f\col\Sigma^2\to\I^3$ be a weakly complete and finite-type ZMC-face, and let $(g,\omega)$ be the Weierstrass data of $f$.
  Then, an end $p$ is embedded if and only if the order of $\omega$ at $p$ is $-2$ or $0$.
\end{proposition}
\begin{proof}
  By Lemma~\ref{lem_residue_omega}, we can take a complex coordinate $z$ centered at $p$ such that $\omega=z^mdz$ holds, where $m$ is an integer different from $-1$.
  Then, if we write $f=(f^0,f^1,f^2)$, we have
  \[f^1+if^2=\int\omega=\frac{1}{m+1}z^{m+1}.\]
  Therefore, it is clear that $p$ is embedded if and only if $m=-2\;\textrm{or}\;0$ holds.
\end{proof}

Therefore, we can show the following Osserman-type inequality.

\begin{theorem}[First Osserman-type inequality]\label{thm_first}
  Let $\bar\Sigma^2$ be a compact Riemann surface and $f\col\bar\Sigma^2\setminus\{p_1,\dots,p_n\}\to\I^3$ a weakly complete and finite-type ZMC-face, and let $(g,\omega)$ be the Weierstrass data of $f$.
  We set $k$ as the number of expanding ends of $f$, and suppose that $p_1,\dots,p_k$ are expanding ends and $p_{k+1},\dots,p_n$ are shrinking ends.
  Then, we have
  \begin{align}\label{eq_first_Osserman_type_inequality_intro}
    -\sum_{j=1}^k\ord_{p_j}(\omega)+\sum_{j=k+1}^n\ord_{p_j}(\omega)\ge 2k.
  \end{align}
  The equality of \eqref{eq_first_Osserman_type_inequality_intro} holds if and only if each end $p_j$ is embedded.
  In particular, any zero mean curvature surface represented as the entire graph of a harmonic polynomial attains equality in \eqref{eq_first_Osserman_type_inequality_intro}.
\end{theorem}
\begin{proof}
  This is an immediate consequence of Propositions~\ref{prop_finite_type_entire_ZMC}, \ref{prop_equivalent_expanding_shrinking} and \ref{prop_embedded_end}.
\end{proof}

One can construct many examples of weakly complete and finite-type ZMC-faces with embedded ends (cf. Examples~\ref{exam_Enneper_surfaces} and \ref{exam_inverse_Enneper_surfaces}).

\section{Asymptotic behaviors of ends}\label{sec_5}

\subsection{Asymptotic behaviors of embedded ends}

In this subsection, we investigate conditions for an embedded end to be asymptotic to that of a plane, a catenoid, or an Enneper paraboloid.
These ends have particularly simple behaviors among embedded ends.
Throughout this section, we denote by $\D^*:=\{z\in\C\setm 0<|z|<1\}$ the punctured unit disk.

\begin{definition}\label{def_asymptotic_behaviors_end}
  Let $f\col\D^*\to\I^3$ be a finite-type ZMC-face which is weakly complete at $0$, and suppose that $0$ is an embedded end.
  We say that $f$ has a \textit{planar end} (resp. a \textit{catenoidal end}, an \textit{Enneper parabolic end}) at 0
  if there exists a map $\tilde f\col\D^*\to\I^3$ such that
  \begin{itemize}
    \item $\tilde f$ is a piece of a plane \eqref{eq_plane} (resp. a catenoid \eqref{eq_catenoid}, an Enneper paraboloid \eqref{eq_Enneper_paraboloid}) by applying an isometry, a homothetic transformation and a dilation with respect to the $t$-axis, and
    \item we have
      \begin{align}\label{eq_asymptotic_behavior}
        f(z)-\tilde f(z)=o(1),
      \end{align}
      where $o(1)$ means that each component tends to 0 as $z\to 0$.
  \end{itemize}
  In particular, in the case of a catenoidal end, we say that $f$ has an \textit{expanding catenoidal end} (resp. a \textit{shrinking catenoidal end}) at $0$ if the first fundamental form $ds^2$ is complete (resp. not complete) at 0.
\end{definition}

The necessary and sufficient condition for an end to be as in Definition~\ref{def_asymptotic_behaviors_end} can be characterized using the Weierstrass data $(g,\omega)$ and the dual Weierstrass data $(g_*,\omega_*)$ as follows:

\begin{proposition}\label{prop_asymptotic_behaviors}
  Let $f\col\D^*\to\I^3$ be a finite-type ZMC-face which is weakly complete at $0$, and suppose that $0$ is an embedded end.
  Then, the following statements hold.
  \begin{enumerate}[{\rm (i)}]
    \item $f$ has a planar end at $0$ if and only if
    \begin{align}\label{eq_planar_end_condition}
      \ord_0(\omega)=-2\quad\textrm{and}\quad\ord_0(\omega_*)\ge 1
    \end{align}
    hold.
    \item $f$ has an expanding catenoidal end at $0$ if and only if
    \begin{align}\label{eq_expanding_catenoidal_end_condition}
      \ord_0(\omega)=-2\quad\textrm{and}\quad\ord_0(\omega_*)=0
    \end{align}
    hold.
    \item $f$ has a shrinking catenoidal end at $0$ if and only if
    \begin{align}\label{eq_shrinking_catenoidal_end_condition}
      \ord_0(\omega)=0\quad\textrm{and}\quad\ord_0(\omega_*)=-2
    \end{align}
    hold.
    \item $f$ has an Enneper parabolic end at $0$ if and only if
    \begin{align}\label{eq_Enneper_type_end_condition}
      \ord_0(\omega)=-2,\qquad\ord_0(\omega_*)=-2\quad\textrm{and}\quad\Res_0(g\omega)=0
    \end{align}
    hold, where $\Res_0(g\omega)$ is the residue of $g\omega$ at $0$.
  \end{enumerate}
\end{proposition}
\begin{proof}
  First suppose that $f=(f^0,f^1,f^2)$ has either planar, expanding catenoidal, shrinking catenoidal or Enneper parabolic end at 0.
  By applying an isometry, a homothetic transformation, a dilation with respect to the $t$-axis to $f$, and changing a coordinate on $\D^*$,
  we may assume that $f$ is asymptotic to one of the following expression.
  Note that $z=re^{i\theta}$ is the polar coordinate.
  \begin{enumerate}[(I)]
    \item\label{item_case_planar_end} In the case of planar end, we have
    \begin{align*}
      f(z)
      &=\left(0,\Re\left(z^{-1}\right),\Im\left(z^{-1}\right)\right)+o(1)\\
      &=\left(0,r^{-1}\cos\theta,-r^{-1}\sin\theta\right)+o(1).
    \end{align*}
    \item\label{item_case_expanding_catenoidal_end} In the case of expanding catenoidal end, we have
    \begin{align*}
      f(z)
      &=\left(\log|z|^{-1},\Re\left(z^{-1}\right),\Im\left(z^{-1}\right)\right)+o(1)\\
      &=\left(-\log r,r^{-1}\cos\theta,-r^{-1}\sin\theta\right)+o(1).
    \end{align*}
    \item\label{item_case_shrinking_catenoidal_end} In the case of shrinking catenoidal end, we have
    \begin{align*}
      f(z)
      &=\bigl(\log|z|,\Re(z),\Im(z)\bigr)+o(1)\\
      &=\bigl(\log r,r\cos\theta,r\sin\theta\bigr)+o(1).
    \end{align*}
    \item\label{item_case_Enneper_type_end} In the case of Enneper parabolic end, we have
    \begin{align*}
      f(z)
      &=\left(\frac{1}{2}\Re\left(z^{-2}\right),\Re\left(z^{-1}\right),\Im\left(z^{-1}\right)\right)+o(1)\\
      &=\left(\frac{1}{2}r^{-2}\cos(2\theta),r^{-1}\cos\theta,-r^{-1}\sin\theta\right)+o(1).
    \end{align*}
  \end{enumerate}
  By the assumption that $f$ has an embedded end at 0, by Proposition~\ref{prop_embedded_end} $\ord_0(\omega)$ is $-2$ or 0.
  Clearly, cases \eqref{item_case_planar_end}, \eqref{item_case_expanding_catenoidal_end}, \eqref{item_case_Enneper_type_end} occur only when $\ord_0(\omega)=-2$, whereas case \eqref{item_case_shrinking_catenoidal_end} occurs only when $\ord_0(\omega)=0$.
  
  Let us consider the cases \eqref{item_case_planar_end}, \eqref{item_case_expanding_catenoidal_end}, \eqref{item_case_Enneper_type_end}.
  We set $m:=\ord_0(g)$, then $g$ and $\omega$ can be expanded around $z=0$ as
  \[g=\sum_{j=m}^\infty a_jz^j,\qquad\omega=\sum_{j=-2}^\infty b_jz^jdz.\]
  By Lemma~\ref{lem_residue_omega}, $b_{-1}=0$ holds.
  Since we have
  \[f^1+if^2=\int\omega=-b_{-2}z^{-1}+o(1),\]
  $b_{-2}=-1$ holds in any of the case \eqref{item_case_planar_end}, \eqref{item_case_expanding_catenoidal_end} or \eqref{item_case_Enneper_type_end}.
  The period condition \eqref{eq_period_condition_ZMC_face} implies that $\Res_0(g\omega)\in\R$, and in particular $a_1\in\R$ when $m=1$.
  Then, $f^0$ can be written as
  \[
    f^0=\Re\int g\omega=\begin{cases}
      o(1) & (m\ge 2),\\
      -a_1\log r+o(1) & (m=1),\\
      -\dfrac{|a_m|}{m-1}r^{m-1}\Bigl(\cos\bigl((m-1)\theta+\alpha\bigr)+o(1)\Bigr) & (m\le 0),
    \end{cases}
  \]
  where $\alpha\in[0,2\pi)$ is a constant.
  Therefore, in case \eqref{item_case_planar_end}, we have $m\ge 2$, namely $\ord_0(\omega_*)=\ord_0(dg)\ge 1$ holds.
  Similarly, case \eqref{item_case_expanding_catenoidal_end} implies that $m=1$, that is, $\ord_0(\omega_*)=0$ holds.
  In case \eqref{item_case_Enneper_type_end}, we have $m=-1$.
  Moreover, in this case, $f^0$ can be written as
  \[f^0=\frac{|a_{-1}|}{2}r^{-2}\cos(2\theta+\alpha)+|a_0|r^{-1}\cos(\theta+\beta)+\Res_0(g\omega)\log r+o(1),\]
  where $\alpha,\beta\in[0,2\pi)$ are constants.
  Therefore, we have $\Res_0(g\omega)=0$.
  
  Next, let us consider the case \eqref{item_case_shrinking_catenoidal_end}.
  As in the previous discussion, we set $m:=\ord_0(g)$.
  Then, $g$ and $\omega$ can be expanded around $z=0$ as
  \[g=\sum_{j=m}^\infty a_jz^j,\qquad\omega=\sum_{j=0}^\infty b_jz^jdz,\]
  and we have $b_0\neq 0$.
  Since we have $\ord_0(g\omega)=m\le -1$ by weak completeness of $f$, we obtain
  \[f^0=\begin{cases}
    a_{-1}b_0\log r+o(1) & (m=-1),\\
    \dfrac{|a_mb_0|}{m+1}r^{m+1}\Bigl(\cos\bigl((m+1)\theta+\alpha\bigr)+o(1)\Bigr) & (m\le -2),
  \end{cases}\]
  where $\alpha\in[0,2\pi)$ is a constant.
  Therefore, we have $m=-1$.
  
  By similar arguments, we can show that the converse is also true.
\end{proof}

\begin{remark}
  In Kato \cite{Kato}, an end satisfying $\ord_0(dF)=-2$ or $-1$ is said to be \textit{simple}, where $F$ is the holomorphic lift of $f$.
  Kato \cite{Kato} also gives a classification of the asymptotic behaviors for simple ends.
\end{remark}

In particular, \eqref{eq_expanding_catenoidal_end_condition} and \eqref{eq_shrinking_catenoidal_end_condition} are in a dual relationship.
Moreover, since we have
\[\Res_0(g\omega)=-\Res_0(g_*\omega_*)\]
by the definition on the dual Weierstrass data $(g_*,\omega_*)$, the condition of \eqref{eq_Enneper_type_end_condition} is in self-dual relationship.
Therefore, we obtain the following propositions.

\begin{proposition}
  Suppose that a ZMC-face $f\col\D^*\to\I^3$ has an expanding catenoidal end {\rm (}resp. a shrinking catenoidal end{\rm)} at $0$.
  Then, the dual map $f_*$ has a shrinking catenoidal end {\rm (}resp. an expanding catenoidal end{\rm)} at $0$.
\end{proposition}

\begin{proposition}
  Suppose that a ZMC-face $f\col\D^*\to\I^3$ has an Enneper parabolic end at $0$.
  Then, the dual map $f_*$ has an Enneper parabolic end at $0$.
\end{proposition}

\subsection{Asymptotic behaviors of shrinking ends}\label{subsec_5_2}

For an expanding end, it follows immediately from \eqref{eq_asymptotic_behavior} that the end is embedded.
On the other hand, there exists a ZMC-face whose end is asymptotic to a shrinking end of a catenoid, but is not embedded.
For example, the ZMC-face determined by the Weierstrass data $(g,\omega)=(1/z^2,2zdz)$ on $\C$, namely the double covering of a catenoid, is such an example.
With this in mind, we introduce the following definition.

\begin{definition}\label{def_layered_shrinking_catenoidal}
  Let $f\col\D^*\to\I^3$ be a finite-type ZMC-face which is weakly complete at $0$.
  We say that $f$ has a \textit{layered shrinking catenoidal end} at 0
  if it satisfies the condition for a shrinking catenoidal end in Definition~\ref{def_asymptotic_behaviors_end}, except that $f$ need not be embedded at $0$.
\end{definition}

\begin{proposition}\label{prop_layered_shrinking_catenoidal_end}
  Let $f\col\D^*\to\I^3$ be a finite-type ZMC-face which is weakly complete at $0$.
  Then, $f$ has a layered shrinking catenoidal end at $0$ if and only if
  \begin{align}\label{eq_shrinking_catenoidal_end_condition_non_embedded}
    \ord_0(\omega)\ge 0\quad\textrm{and}\quad\ord_0(g\omega)=-1
  \end{align}
  hold.
\end{proposition}
\begin{proof}
  This proposition can be shown by using the same discussion in the proof of Theorem~\ref{prop_asymptotic_behaviors}.
  Suppose that $f=(f^0,f^1,f^2)$ has a layered shrinking catenoidal end at $0$.
  Since $0$ is a shrinking end, we have $\ord_0(\omega)\ge 0$.
  In particular, $f^1$ and $f^2$ converge to finite values as $z\to 0$, which implies that we do not need to take them into account when considering the asymptotic behavior.
  We set $m:=\ord_0(g)$ and $l:=\ord_0(\omega)$, then $g$ and $\omega$ can be expanded around $z=0$ as
  \[g=\sum_{j=m}^\infty a_jz^j,\qquad\omega=\sum_{j=l}^\infty b_jz^jdz.\]
  By weak completeness of $f$, we have $\ord_0(g\omega)=m+l\le -1$.
  Then, $f^0$ can be written as
  \[f^0=\begin{cases}
      a_mb_l\log r+o(1) & (m+l=-1),\\
      \dfrac{|a_mb_l|}{m+l+1}r^{m+l+1}\Bigl(\cos\bigl((m+l+1)\theta+\alpha\bigr)+o(1)\Bigr) & (m+l\le -2),
    \end{cases}\]
  where $\alpha\in[0,2\pi)$ is a constant.
  Therefore, we have $m+l=-1$.
  The converse can also be shown easily.
\end{proof}

\section{Proof of the Osserman-type inequalities}\label{sec_6}

In this section, we show the second and third Osserman-type inequalities.
We first state the second Osserman-type inequality.

\begin{theorem}[Second Osserman-type inequality]\label{thm_second}
  In addition to the setting of Theorem~\ref{thm_third}, we denote by $S_f=\{q_1,\dots,q_m\}$ the set of singular points of $f$ and we set $d_j:=\ord_{q_j}(\omega)$ for each $j=1,\dots,m$.
  Then, the Gauss map $g\col\bar\Sigma^2\to\C\cup\{\infty\}$ and the dual Gauss map $g_*\col\bar\Sigma^2\to\C\cup\{\infty\}$ of $f$ satisfy
  \begin{align}\label{eq_second_Osserman_type_inequality_intro}
    \deg(g)+\deg(g_*)\ge n+d_1+\dots+d_m.
  \end{align}
  The equality of \eqref{eq_second_Osserman_type_inequality_intro} holds if and only if each end $p_j$ is either planar, expanding catenoidal, or shrinking catenoidal.
  In particular, if $f$ is represented as an entire graph, then the equality condition of \eqref{eq_second_Osserman_type_inequality_intro} is equivalent that the image of $f$ is a plane in $\I^3$.
\end{theorem}

We set $\Sigma^2:=\bar\Sigma^2\setminus\{p_1,\dots,p_n\}$.
Throughout this section, we use the notations of Theorems~\ref{thm_third} and \ref{thm_second}, and suppose that $p_1,\dots,p_k$ are expanding ends and $p_{k+1},\dots,p_n$ are shrinking ends.
To prove the theorems, we prepare the following lemma.

\begin{lemma}\label{lem_deg_g_detail}
  The Gauss map $g$ satisfies
  \begin{align}\label{ineq_deg_g_detail}
    \deg(g)\ge n-k+\sum_{j=k+1}^n\ord_{p_j}(\omega)+\sum_{j=1}^md_j.
  \end{align}
  The equality of \eqref{ineq_deg_g_detail} holds if and only if 
  \begin{itemize}
    \item $\ord_{p_j}(\omega_*)\ge 0$ holds for each expanding end $p_j$, and
    \item $\ord_{p_j}(g\omega)=-1$ holds for each shrinking end $p_j$.
  \end{itemize}
\end{lemma}
\begin{proof}
  $\deg(g)$ is equal to the number of the poles, counting multiplicities.
  Since $g$ is holomorphic on $\Sigma^2\setminus S_f$ and $\ord_{q_j}(g)=-\ord_{q_j}(\omega)$ holds for each singular point $q_j$, we have
  \begin{align}\label{eq_mappnig_degree}
    \deg(g)=\sum_{j=1}^n\max\left(0,-\ord_{p_j}(g)\right)+\sum_{j=1}^md_j.
  \end{align}
  Considering expanding ends $p_1,\dots,p_k$, we clearly have
  \begin{align}\label{ineq_deg_g_expanding}
    \sum_{j=1}^k\max\left(0,-\ord_{p_j}(g)\right)\ge 0,
  \end{align}
  and the equality holds if and only if $g$ does not have a pole at each expanding end $p_j$, that is, $\ord_{p_j}(\omega_*)=\ord_{p_j}(dg)\ge 0$ holds.
  Moreover, for shrinking ends $p_{k+1},\dots,p_n$, weak completeness of $f$ implies that $g\omega$ has a pole at these points.
  Since $\omega$ is holomorphic at shrinking ends, we have
  \[-\ord_{p_j}(g)\ge 1+\ord_{p_j}(\omega)\ge 1.\]
  Therefore, we obtain
  \begin{align}\label{ineq_deg_g_shrinking}
    \sum_{j=k+1}^n\max\left(0,-\ord_{p_j}(g)\right)=-\sum_{j=k+1}^n\ord_{p_j}(g)\ge n-k+\sum_{j=k+1}^n\ord_{p_j}(\omega),
  \end{align}
  and the equality holds if and only if $g\omega$ has a simple pole at shrinking ends.
  Therefore, by combining inequalities \eqref{ineq_deg_g_expanding} and \eqref{ineq_deg_g_shrinking} with \eqref{eq_mappnig_degree}, we obtain the assertion.
\end{proof}

\begin{corollary}\label{cor_deg_g}
  The Gauss map $g$ satisfies
  \begin{align}\label{ineq_deg_g}
    \deg(g)\ge n-k+\sum_{j=1}^md_j.
  \end{align}
  The equality of \eqref{ineq_deg_g} holds if and only if 
  \begin{itemize}
    \item $\ord_{p_j}(\omega_*)\ge 0$ holds for each expanding end $p_j$, and
    \item $\ord_{p_j}(\omega)=0$ and $\ord_{p_j}(\omega_*)=-2$ hold for each shrinking end $p_j$.
  \end{itemize}
\end{corollary}
\begin{proof}
  \eqref{ineq_deg_g} follows immediately from Lemma~\ref{lem_deg_g_detail}.
  The equality condition of \eqref{ineq_deg_g} is given by that of \eqref{ineq_deg_g_detail} together with $\ord_{p_j}(\omega)=0$ for each shrinking end $p_j$.
  For shrinking ends, this is equivalent to $\ord_{p_j}(\omega)=0$ and $\ord_{p_j}(g)=-1$, that is, $\ord_{p_j}(\omega_*)=-2$.
\end{proof}

On the other hand, we obtain the following inequality for the degree of the dual Gauss map $g_*$.

\begin{proposition}\label{prop_deg_dual_g}
  The dual Gauss map $g_*$ satisfies
  \begin{align}\label{ineq_deg_dual_g}
    \deg(g_*)\ge k.
  \end{align}
  The equality of \eqref{ineq_deg_dual_g} holds if and only if each expanding end $p_j$ is embedded, namely $\ord_{p_j}(\omega)=-2$ holds.
\end{proposition}
\begin{proof}
  Since we have $\ord_{p_j}(\omega)\le-2$ for each expanding end $p_j$ by Proposition~\ref{prop_equivalent_expanding_shrinking}, the dual Gauss map $g_*$ has a pole at $p_j$.
  Therefore, we obtain \eqref{ineq_deg_dual_g}.
  Since $g_*$ is holomorphic on $\bar\Sigma^2\setminus\{p_1,\dots,p_k\}$, the equality condition is clear by Proposition~\ref{prop_embedded_end}.
\end{proof}

With these preparations, we now prove Theorems~\ref{thm_second} and \ref{thm_third}.

\begin{proof}[Proof of Theorem~\ref{thm_second}]
  By combining Corollary~\ref{cor_deg_g} and Proposition~\ref{prop_deg_dual_g}, we obtain \eqref{eq_second_Osserman_type_inequality_intro},
  and the equality of \eqref{eq_second_Osserman_type_inequality_intro} holds if and only if 
  \begin{itemize}
    \setlength{\leftskip}{-1em}
    \item $\ord_{p_j}(\omega)=-2$ and $\ord_{p_j}(\omega_*)\ge 0$ hold for each expanding end $p_j$, and
    \item $\ord_{p_j}(\omega)=0$ and $\ord_{p_j}(\omega_*)=-2$ hold for each shrinking end $p_j$.
  \end{itemize}
  By Proposition~\ref{prop_asymptotic_behaviors}, the above condition means that each end of $f$ is either planar, expanding catenoidal, or shrinking catenoidal.
  In the case where $f$ is represented as an entire graph, the equality condition follows immediately from Proposition~\ref{prop_finite_type_entire_ZMC}.
\end{proof}

\begin{proof}[Proof of Theorem~\ref{thm_third}]
  Applying the Riemann-Roch theorem to $\omega$ on $\bar\Sigma^2$, we have
  \begin{align}\label{eq_Riemann_Roch_omega}
    \sum_{j=1}^n\ord_{p_j}(\omega)+\sum_{j=1}^md_j=-\chi(\bar\Sigma^2).
  \end{align}
  By \eqref{ineq_deg_g_detail} and \eqref{eq_Riemann_Roch_omega}, we obtain
  \begin{align}\label{ineq_deg_g_detail_2}
    \deg(g)\ge n-k-\chi(\bar\Sigma^2)-\sum_{j=1}^k\ord_{p_j}(\omega).
  \end{align}
  By Proposition~\ref{prop_equivalent_expanding_shrinking}, we have $\ord_{p_j}(\omega)\le-2$ for each $j=1,\dots,k$.
  Combining this with \eqref{ineq_deg_g_detail_2} yields \eqref{eq_third_Osserman_type_inequality_intro}.
  Moreover, the equality of \eqref{eq_third_Osserman_type_inequality_intro} holds if and only if 
  \begin{itemize}
    \setlength{\leftskip}{-1em}
    \item $\ord_{p_j}(\omega)=-2$ and $\ord_{p_j}(\omega_*)\ge 0$ hold for each expanding end $p_j$, and
    \item $\ord_{p_j}(g\omega)=-1$ holds for each shrinking end $p_j$.
  \end{itemize}
  By Propositions~\ref{prop_asymptotic_behaviors} and \ref{prop_layered_shrinking_catenoidal_end}, the above condition means that each end of $f$ is either planar, expanding catenoidal, or layered shrinking catenoidal.
  The argument in the case where $f$ can be represented as an entire graph the same as in the proof of Theorem~\ref{thm_second}.
\end{proof}

\section{The maximal number of omitted values of Gauss maps}\label{sec_7}

In this section, as an application of the third Osserman-type inequality \eqref{eq_third_Osserman_type_inequality_intro}, we give an estimate for the maximal number of omitted values of Gauss maps for weakly complete and finite-type ZMC-faces.
First, we show the following lemma.

\begin{lemma}\label{lem_at_least_one_expanding_ends}
  A weakly complete and finite-type ZMC-face has at least one expanding end.
\end{lemma}
\begin{proof}
  Let $(g,\omega)$ be the Weierstrass data of a weakly complete and finite-type ZMC-face $f$.
  Assume that $f$ has no expanding ends.
  Then, $\omega$ is holomorphic at each end, so the dual Gauss map $g_*$ is a holomorphic function on a compact Riemann surface $\bar\Sigma^2$.
  Therefore, $g_*$ is constant and $\omega\equiv 0$, which yield a contradiction to the compatibility condition of $(g,\omega)$.
\end{proof}

Next, we give the following estimate for the number of omitted values of the Gauss map $g$.

\begin{lemma}[\cite{Kawakami}]\label{lem_Kawakami}
  Let $f\col\Sigma^2\to\I^3$ be a weakly complete and finite-type ZMC-face, and let $g\col\Sigma^2\to\C\cup\{\infty\}$ be the Gauss map of $f$.
  We assume that $g$ is non-constant and denote by $D_g$ the number of omitted values of $g$.
  Then, we have
  \[D_g\le 3-\frac{k}{\deg(g)},\]
  where $k$ is the number of expanding ends of $f$.
\end{lemma}
\begin{proof}
  The proof of this lemma is due to a suggestion of Yu Kawakami.
  We use an argument similar to that in Kawakami-Kobayashi-Miyaoka~\cite{KKM2008}.
  
  Let $\{\alpha_1,\dots,\alpha_{D_g}\}\subseteq\C\cup\{\infty\}$ be the set of omitted values of $g$.
  We define $b_0$ as the sum of the branching orders of $g$ at these omitted values.
  Then, we have
  \[b_0\ge D_g\deg(g)-n.\]
  Here, we denote by $b_g$ the total branching order of $g$.
  Applying the Riemann-Hurwitz theorem to $g$ on $\bar\Sigma^2$, we have
  \[b_g=2\deg(g)-\chi(\bar\Sigma^2).\]
  Since $b_0\le b_g$ holds, we get
  \[D_g\le\frac{b_0+n}{\deg(g)}\le\frac{b_g+n}{\deg(g)}=2+\frac{-\chi(\bar\Sigma^2)+n}{\deg(g)}.\]
  By using the third Osserman-type inequality \eqref{eq_third_Osserman_type_inequality_intro}, we obtain
  \[D_g\le 2+\frac{\deg(g)-k}{\deg(g)}=3-\frac{k}{\deg(g)}.\]
\end{proof}

Consequently, we obtain the following theorem.

\begin{theorem}\label{thm_Dg}
  The maximal number of omitted values of Gauss maps for weakly complete and finite-type ZMC-faces, excluding a plane, is $2$.
\end{theorem}
\begin{proof}
  By Lemmas~\ref{lem_at_least_one_expanding_ends} and \ref{lem_Kawakami}, we have $D_g\le 2$.
  Since the Gauss map of a catenoid in $\I^3$ (cf. Example~\ref{exam_catenoid}) has two omitted values, this bound is sharp.
\end{proof}

\begin{remark}
  If $f$ is a finite-type zero mean curvature surface represented as an entire graph, excluding a plane, the Gauss map $g$ is a polynomial by Proposition~\ref{prop_finite_type_entire_ZMC}.
  Therefore, in this case, $g$ has exactly one omitted value $\infty$.
\end{remark}

\section{Examples}\label{sec_8}

In this section, we present some examples of ZMC-faces that attain equalities in the Osserman-type inequalities \eqref{eq_first_Osserman_type_inequality_intro}, \eqref{eq_second_Osserman_type_inequality_intro} or \eqref{eq_third_Osserman_type_inequality_intro}.
First, we give examples attaining equality in the first Osserman-type inequality \eqref{eq_first_Osserman_type_inequality_intro}.

\begin{example}[Enneper surface of order $m$]\label{exam_Enneper_surfaces}
  Let $\Sigma^2:=\C$ and $(g,\omega):=(z^{m-1},dz)$ for a positive integer $m\ge 2$.
  Then, the corresponding ZMC-face $f$ has an expanding embedded end at $z=\infty$.
  Note that this surface is given in Akamine-Lee-Yang~\cite[Example 3.7]{ALYpreprint} and is called the \textit{Enneper surface of order $m$}.
\end{example}
\begin{example}[Inverse Enneper surface of order $m$]\label{exam_inverse_Enneper_surfaces}
  Let $\Sigma^2:=\C\setminus\{0\}$ and $(g,\omega):=(z^{-m-1},dz)$ for a positive integer $m\ge 1$.
  Then, the corresponding ZMC-face $f$ has an expanding embedded end at $z=\infty$ and a shrinking embedded end at $z=0$ (see Figure~\ref{fig_first_Osserman_ineq}).
  Although the figure seems to exhibit multiple shrinking ends, note that there is in fact only one shrinking end.
  We call this surface \textit{inverse Enneper surface of order $m$}.
  These surfaces are related to Example~\ref{exam_Enneper_surfaces} by inversion.
  More precisely, let us consider the inversion with respect to the $t$-axis, namely the bijection
  \[\I^3\setminus L\ni(t,x,y)\mapsto\left(t,\frac{x}{x^2+y^2},\frac{y}{x^2+y^2}\right)\in\I^3\setminus L,\]
  where $L:=\{(t,x,y)\in\I^3\setm x=y=0\}$.
  For $m\ge 2$, this map takes the Enneper surface of order $m$ to the inverse Enneper surface of order $m$.
  Moreover, this map takes the inverse Enneper surface of order 1 to a plane which is not parallel to the $xy$-plane.
  
  Here, let us consider a map $\Phi_m\col\I^3\to\R$ defined by
  \[\Phi_m(t,x,y):=mt(x^2+y^2)^m+\Re\bigl((x+iy)^m\bigr).\]
  Then, the inverse Enneper surface of order $m$ can be represented implicitly as
  \[\Bigl\{(t,x,y)\in\I^3\setminus L\setm\Phi_m(t,x,y)=0\Bigr\}.\]
\end{example}

\begin{figure}[htbp]
  \centering
  \includegraphics[width=5cm]{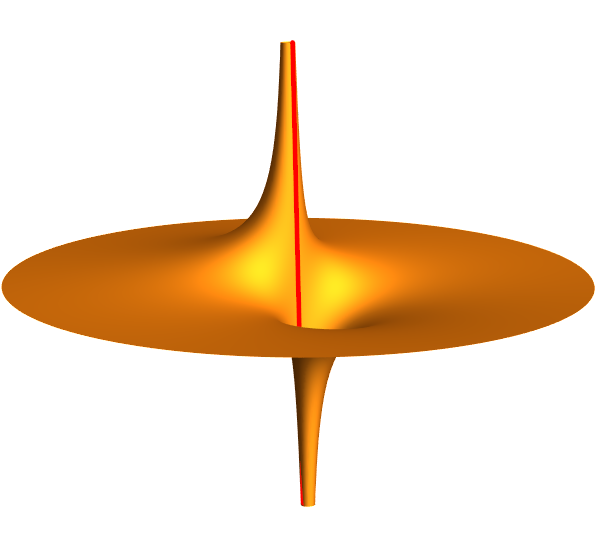}
  \caption{Inverse Enneper surface of order 1. The red line denotes $L$.}
  \label{fig_first_Osserman_ineq}
\end{figure}

Next, we provide examples which satisfy the equality condition of the second Osserman-type inequality \eqref{eq_second_Osserman_type_inequality_intro} or the third Osserman-type inequality \eqref{eq_third_Osserman_type_inequality_intro}.

\begin{example}\label{exam_two_catenoidal_ends}
  Let $\Sigma^2:=\C\setminus\{0\}$ and
  \begin{align}\label{eq_exam_two_catenoidal_ends}
    (g,\omega):=\left(\frac{z}{z^2-1},\frac{z^2-1}{z^2}dz\right).
  \end{align}
  This Weierstrass data gives the weakly complete and finite-type ZMC-face $f\col\Sigma^2\to\I^3$, which has two expanding catenoidal ends at $z=0,\infty$.
  Moreover, $f$ has two singular points, which are cross caps, at $z=\pm 1$.
  Since the dual Weierstrass data of $f$ is given by
  \[(g_*,\omega_*)=\left(\frac{z^2+1}{z},-\frac{z^2+1}{(z^2-1)^2}dz\right),\]
  we have $\deg(g)=\deg(g_*)=2$.
  Therefore, this example attains equalities in the first, second and third Osserman-type inequalities.
  Writing $f=(f^0,f^1,f^2)$, we have
  \[f^0(z)=\log|z|,\]
  and the limit of $f^0$ when $z\to 0$ (resp. $z\to\infty$) is $-\infty$ (resp. $+\infty$).
  Therefore, the ZMC-face $f$ is properly embedded outside a compact subset (see Figure~\ref{fig_intro} (a)).
\end{example}

\begin{example}\label{exam_same_ZMC_1}
  Let $\Sigma^2:=\C\setminus\{0,1\}$ and
  \begin{align}\label{eq_same_Weierstrass_data_1}
    (g,\omega):=\left(\frac{z^2}{(z^2+1)(z-1)},\frac{z^2+1}{z^2}dz\right).
  \end{align}
  This Weierstrass data gives the weakly complete and finite-type ZMC-face, which has a planar end at $z=0$ and an expanding catenoidal end at $z=\infty$ and a shrinking catenoidal end at $z=1$.
  Therefore, the ZMC-face given by the Weierstrass data \eqref{eq_same_Weierstrass_data_1} satisfies the equality conditions of all three Osserman-type inequalities.
  If we write $f=(f^0,f^1,f^2)$, we have
  \[f^0(z)=\log|z-1|,\]
  and the limit of $f^0$ when $z\to 0$ (resp. $z\to\infty$) is $0$ (resp. $+\infty$).
  This implies that the ZMC-face $f$ is properly embedded outside a compact subset (see Figure~\ref{fig_intro} (b)).
\end{example}

\begin{example}\label{exam_same_ZMC_2}
  Let $\Sigma^2:=\C\setminus\{0,1\}$ and
  \begin{align}\label{eq_same_Weierstrass_data_2}
    (g,\omega):=\left(\frac{z^2}{(z^2-1)(z-1)},\frac{z^2-1}{z^2}dz\right).
  \end{align}
  This Weierstrass data gives the weakly complete and finite-type ZMC-face, which has a planar end at $z=0$ and an expanding catenoidal end at $z=\infty$ and a layered shrinking catenoidal end, which is not embedded, at $z=1$.
  Therefore, the ZMC-face given by the Weierstrass data \eqref{eq_same_Weierstrass_data_2} satisfies the equality condition only for the third Osserman-type inequality (see Figure~\ref{fig_intro} (c)).
\end{example}

\begin{remark}
  In Examples~\ref{exam_same_ZMC_1} and \ref{exam_same_ZMC_2}, not only the domain $\Sigma^2$ but also the numbers of ends and expanding ends coincide.
  Moreover, the mapping degrees of the Gauss maps of \eqref{eq_same_Weierstrass_data_1} and \eqref{eq_same_Weierstrass_data_2} coincide, which implies that the total curvatures with respect to the lift-metrics also coincide.
\end{remark}

Using elliptic functions, we can also construct examples for which the genus of $\bar\Sigma^2$ is 1.

\begin{example}\label{exam_genus_1}
  We set $T:=\C/(\Z+\Z i)$ as the square torus and we denote by $\wp(z)$ the Weierstrass $\wp$-function on $T$.
  We set
  \[\omega:=d\left(\frac{\wp'}{\wp}\right)=\frac{\wp''\wp-(\wp')^2}{\wp^2}dz,\qquad g:=\frac{dz}{\omega}\frac{\wp'}{\wp-c},\]
  where $c\in\C$ is a constant, and we define $\Sigma^2:=T\setminus\{\mathrm{the\;poles\;of\;}\omega\mathrm{\;or\;}g\omega\}$.
  Then, we have
  \[\Re\oint_\gamma g\omega=\Re\oint_\gamma\frac{\wp'}{\wp-c}dz=\Re\oint_\gamma d\bigl(\log(\wp-c)\bigr)=0\]
  for any closed curve $\gamma$ on $\Sigma^2$.
  Moreover, we can check that $\omega$ and $g\omega$ have no common zeros on $\Sigma^2$.
  Therefore, $(g,\omega)$ satisfies the compatibility and period conditions and gives a weakly complete and finite-type ZMC-face $f$ defined on $\Sigma^2$.
  
  In the case of $c=0$, the ZMC-face $f$ is defined on $\Sigma^2=T\setminus\{0,(1+i)/2\}$, and $f$ has two expanding catenoidal ends at $z=0,(1+i)/2$ and four cross caps at
  \begin{align}\label{eq_singular_points_genus_1}
    z=\frac{1+i}{4},\frac{1-i}{4},\frac{-1+i}{4},\frac{-1-i}{4}.
  \end{align}
  Since we have
  \[f^0(z)=\log|\wp(z)|,\]
  where $f=(f^0,f^1,f^2)$, the limit of $f^0$ when $z\to 0$ (resp. $z\to (1+i)/2$) is $+\infty$ (resp. $-\infty$).
  Therefore, $f$ is properly embedded outside a compact subset  (see Figure~\ref{fig_intro} (d)).
  
  On the other hand, in the case of $c=\wp(1/4)$ the ZMC-face $f$ is defined on $\Sigma^2=T\setminus\{0,(1+i)/2,-1/4,1/4\}$.
  Then, $f$ has an expanding catenoidal end at $z=0$ and a planar end at $z=(1+i)/2$ and two shrinking catenoidal ends at $z=-1/4,\,1/4$.
  In addition, $f$ has four cross caps at points given by \eqref{eq_singular_points_genus_1}.
  Since we have
  \[f^0(z)=\log|\wp(z)-\wp(1/4)|,\qquad f^1(z)+if^2(z)=\frac{\wp'(z)}{\wp(z)},\]
  the limit of $f^0$ when $z\to 0$ (resp. $z\to (1+i)/2$) is $+\infty$ (resp. $\log|\wp(1/4)|$).
  Moreover, we can check that the limit of $(f^1,f^2)$ as $z\to -1/4$ differs from that as $z\to 1/4$.
  Therefore, even in the case $c=\wp(1/4)$, the ZMC-face $f$ is properly embedded outside a compact subset (see Figure~\ref{fig_intro} (e)).
\end{example}

\begin{remark}
  In this paper, we present several examples with shrinking ends as well as expanding ends.
  By contrast, Kato \cite{Kato} investigates expanding catenoidal ends and obtains a general description of the Weierstrass data for zero mean curvature surfaces with singularities whose ends are all of this type.
  Since this general description allows branch points, namely singular points of corank 2, it covers a broader class than that of ZMC-faces.
  Example~\ref{exam_two_catenoidal_ends} and Example~\ref{exam_genus_1} with $c=0$ belong to the class described in \cite{Kato}.
  As shown above, these two examples are particularly nice, since they are properly embedded outside a compact subset.
\end{remark}

\appendix

\section{A criterion for cross caps of ZMC-faces}\label{app_A}

In this appendix, we prove Proposition~\ref{prop_criterion_cross_cap}.
We use the following criterion on cross caps.

\begin{fact}[Whitney~\cite{Whitney1944}]\label{fact_criterion_cross_cap}
  Let $U$ be a domain in $\R^2$, and let $f\col U\to\R^3$ be a smooth map.
  We fix a singular point $p\in U$ of $f$.
  We denote by $(u,v)$ the standard coordinate of $\R^2$, and assume that $f_v(p)=0$ holds.
  Then, $f$ has a cross cap at $p$ if and only if
  \[\det(f_u,f_{uv},f_{vv})\neq 0\]
  holds at $p$.
\end{fact}

\begin{proof}[Proof of Proposition~\ref{prop_criterion_cross_cap}]
  Let $(U,z)$ be a complex coordinate of $\Sigma^2$ centered at a singular point $p$.
  We define two vector fields $\xi$ and $\eta$ defined on $U$ as
  \[\xi:=\frac{1}{g\hat\omega}\frac{\partial}{\partial z}+\overline{\left(\frac{1}{g\hat\omega}\right)}\frac{\partial}{\partial\bar z},\qquad\eta:=\frac{i}{g\hat\omega}\frac{\partial}{\partial z}+\overline{\left(\frac{i}{g\hat\omega}\right)}\frac{\partial}{\partial\bar z}.\]
  By a direct calculation, we obtain
  \[\xi f=\bigl(1,\Re(1/g),\Im(1/g)\bigr),\qquad\eta f=\bigl(0,-\Im(1/g),\Re(1/g)\bigr).\]
  Since $g$ has a pole at $p$, we have $(\eta f)(p)=0$.
  Moreover, we can calculate that
  \begin{align*}
    \xi\eta f&=\left(0,\Im\left(\frac{g_z}{g^3\hat\omega}\right),-\Re\left(\frac{g_z}{g^3\hat\omega}\right)\right),\\
    \eta\eta f&=\left(0,\Re\left(\frac{g_z}{g^3\hat\omega}\right),\Im\left(\frac{g_z}{g^3\hat\omega}\right)\right).
  \end{align*}
  Therefore, we have
  \[\Delta:=\det\bigl(\xi f,\xi\eta f,\eta\eta f\bigr)=\left|\frac{g_z}{g^3\hat\omega}\right|^2=\left|\frac{d(1/g)}{g\omega}\right|^2.\]
  Since $g\omega$ does not have zero at $p$ by the compatibility condition, $\Delta$ has non-zero value at $p$ if and only if $\ord_p(g)=-1$ holds, that is, $\ord_p(\omega)=1$ holds.
  Therefore, by Fact~\ref{fact_criterion_cross_cap} the assertion follows.
\end{proof}

\bibliographystyle{amsplain}
\bibliography{reference}

\end{document}